\documentclass[12pt]
{amsart}

\usepackage{cite}
\usepackage{amssymb}
\usepackage{amsmath}
\usepackage{amsthm}
\usepackage{amsfonts}
\usepackage{graphicx}%
\usepackage{comment}
\usepackage{bbm}
\usepackage{xcolor}
\usepackage[toc,page]{appendix}
\usepackage{mathtools}
\usepackage[utf8]{inputenc}
\usepackage{enumitem}   

\usepackage[pagebackref]{hyperref}
\usepackage[margin=1.3in, marginpar=1.3in, marginparsep=.1in]{geometry}
\usepackage{hyperref}

\usepackage{tikz-cd}

\numberwithin{equation}{section}

\theoremstyle{plain} 
\newtheorem{thm}{Theorem}[section]

\newtheorem{prop}[thm]{Proposition} 

\newtheorem{definition}[thm]{Definition}
\newtheorem{rmk}[thm]{Remark} 
 
\newtheorem*{acknowledgement}{Acknowledgement}

\newcommand{\R}{\mathbb{R}}  
\newcommand{\E}{\mathbb{E}} 
\newcommand{\Prob}{\mathbb{P}}
\newcommand{\Hei}{\mathbb{H}}
\newcommand{\C}{\mathbb{C}}

\title[Berry-Heisenberg random waves]{Berry-Heisenberg random waves}

\author[Carfagnini]{Marco Carfagnini}
\address{School of Mathematics and Statistics\\
University of Melbourne \\
Parkville VIC 3010, Australia}
\email{marco.carfagnini@unimelb.edu.au}

\author[Todino]{Anna Paola Todino}
\address{ Dipartimento di Scienze e Innovazione Tecnologica\\
Universit\`a del Piemonte Orientale\\
Italy}
\email{annapaola.todino@uniupo.it}



\begin{document}
\begin{abstract}
We construct a new family of random fields on the Heisenberg group $\mathbb{H}$, the sub-Riemannian analog of $\R^{n}$.  These fields are generalized random eigenfunctions of the sub-Laplacian on $\Hei$, and can be viewed as the sub-Riemannian counterpart to the Berry random wave model in $\mathbb{R}^{n}$. The construction of such waves relies on the representation theory of $\Hei$, and differs from the Euclidean case because of the presence of infinite-dimensional unitary irreducible representations. This work represents a first step towards studying random waves and their geometry in sub-Riemannian spaces.
\end{abstract}

\maketitle

\tableofcontents

\section{Introduction and main results}
Random wave models provide a fundamental link between spectral geometry, quantum chaos, and the theory of Gaussian random fields. Following the seminal ansatz introduced by Berry \cite{Berry2002,Berry1977}, chaotic quantum eigenstates in high-energy regimes are widely conjectured to exhibit universal statistical behaviors, locally modeled by Gaussian monochromatic plane waves. This perspective has stimulated an extensive literature on the geometry and topology of random eigenfunctions, including nodal sets, excursion sets, and related geometric functionals; see e.g. \cite{RevModPhys.89.045005,AdlerTaylor07}.
Most existing works concern Euclidean spaces and compact Riemannian manifolds, see for instance the surveys \cite{canzani_MRWsurvey, igorsurvey}, and, more recently, negatively curved manifolds and hyperbolic spaces \cite{GrottoPeccati2023}. In all these settings, random waves arise naturally from the spectral theory of the Laplace operator and can be analyzed through explicit harmonic-analytic representations. On the other hand, the study of such random waves on sub-Riemannian structures remains largely unexplored.
The Heisenberg group $\Hei$ plays a central role in geometric analysis, harmonic analysis, and mathematical physics, and it represents the flat model space in sub-Riemannian geometry. 
The aim of this paper is to construct a new class of Gaussian fields on $\mathbb{H}$, which are referred to as \textit{Berry-Heisenberg random waves}. Our construction relies on harmonic analysis for the sub-Laplacian $\mathcal{L}$ and representation theory for $\Hei$. These fields play  a role analogous to that of Berry's monochromatic waves in Euclidean geometry. We prove explicit spectral representations, establish their basic probabilistic and analytic properties, and compare them with classical random wave models on $\R^{n}$.

\subsection{Berry random wave model and Riemannian random waves}\label{sec:berry}
Broadly speaking, a Riemannian Random Wave model (RW) is a random, Gaussian superposition of Laplacian eigenfunctions on a manifold. The foundational benchmark for this framework is the complex-valued planar random wave, introduced by Berry in his seminal paper \cite{Berry1977} to model isotropic, high-energy monochromatic quantum states. 
Formally, for a fixed wavenumber parameter $\alpha > 0$, Berry's complex model on $\R^{n}$ is defined as a centered, stationary, and isotropic Gaussian random field $\{ b^{\C}_{\alpha} (x) , \, x\in \R^{n}\}$ whose trajectories are almost surely  (generalized) Gaussian  eigenfunctions for the Laplace operator $\Delta_{\R^{n}}$, that is, it satisfies the Helmholtz equation
\begin{align}\label{eqn.intro}
    &\Delta_{\R^{n}} \,  b^{\C}_{\alpha} (x) = - \alpha \, b^{\C}_{\alpha} (x) , \quad a.s.
\end{align}
Note that the spectrum of $-\Delta_{\R^{n}}$ in $L^{2}(\R^{n})$ is $[0,\infty)$ and it is purely continuous, containing no eigenvalues. For this reason, the function $b^{\C}_{\alpha} (x)$ is referred to as a \emph{generalized} eigenfunction, as it satisfies \eqref{eqn.intro} but is not in $L^{2}(\R^{n})$.
This field arises as the limit in distribution of a finite random superposition of independent Euclidean plane waves, all with the same wavenumber $\sqrt{\alpha}$ but different directions. For instance, in dimension $n=2$, $\{ b^{\C}_{\alpha} (x) , \, x\in \R^{n}\}$ is the limit (in distribution) of
\begin{equation}\label{eq:supBerry}
    \frac{\sqrt{2}}{\sqrt N}\sum_{i=1}^N \cos (\sqrt{\alpha} \, x\cdot \theta_i+\phi_i) + i \frac{\sqrt{2}}{\sqrt N}\sum_{i=1}^N \sin (\sqrt{\alpha} \, x\cdot \theta_i+\phi_i), \quad N\to\infty,
\end{equation}
where $\{\theta_i\}_{i=1}^N$ and $\{\phi_i\}_{i=1}^N$ are i.i.d. uniform random variables in $[0,2\pi]$.

The field $\{ b^{\C}_{\alpha} (x) , \, x\in \R^{n}\}$ is unique in the following sense.  If $Y$ is an isotropic centered complex-valued Gaussian random field on the plane with variance 2 and such that $  \Delta Y + \alpha \, Y =0$, then $Y$ has the same distribution as $b^{\C}_{\alpha}$, see \cite[Theorem 5.7.2]{AdlerTaylor07}.
Being a mean-zero Gaussian field, the law of $b^{\C}_{\alpha}$ is uniquely characterized by its covariance function, which is given by 
\begin{align}\label{eq:covBerry}
   \E\left[ b^{\C}_{\alpha} (x) \overline{b^{\C}_{\alpha}(y)} \right] = 2j_{n} (\sqrt{\alpha} \|x-y\|)  := \frac{\nu !\, 2^{\nu+1}}{|\sqrt{\alpha} \|x-y\||^{\nu}}J_{\nu} (\vert \sqrt{\alpha}\, \|x-y\| \vert ) , \quad \nu = \frac{n}{2}-1,
\end{align}
where $J_{\nu}$ denotes the Bessel function of the first kind of order $\nu$. For two-dimensional waves, this reduces to $2J_0(\sqrt{\alpha}\|x-y\|)$.
Let us consider both an analytic and algebraic heuristic on the uniqueness of the field. Analytically, the covariance of $b^{\C}_{\alpha}$ is the unique smooth radial solution to the equation
\begin{align}\label{eqn.intro.idk}
    &\Delta_{\R^{n}} f = - \alpha \, f, & f(0)=2.
\end{align}
Algebraically, the uniqueness of the field is related to the unitary irreducible representations of the additive group $\R^{n}$ being all one-dimensionals, see also Remark \ref{rmk.uniqueness} for more details.
In the context of quantum chaos theory, concerning the relation between classical chaotic systems and their quantum mechanical counterpart, Berry \cite{Berry1977} suggested that the behavior of high energy ($\alpha \to \infty$) \emph{deterministic} eigenfunctions $f_\alpha$ should be universal, at least on generic chaotic surfaces, meaning that one can compare such an eigenfunction $f_{\alpha}$ of large eigenvalue $\alpha$ with the random
monochromatic plane wave of wavenumber $\alpha$.
This fundamental insight sparked a broad line of research devoted to the study of the geometric and topological functionals of the Berry model; 
we refer for instance to \cite{NourdinPeccatiRossi19,DalmaoEstradeLeon,GMT,Notarnicola}.

Berry-like waves have been extended to general Riemannian manifolds $(\mathcal{M},g)$. When $\mathcal{M}$ is compact, the spectrum of the Laplace-Beltrami operator $\Delta_g$ is discrete, $0 = \lambda_0 < \lambda_1 \leq \lambda_2 \dots \uparrow \infty$. To replicate Berry's model in this setting, 
Zelditch \cite{Zelditch2009} introduced the notion of \emph{monochromatic random waves}, defined as random superpositions of Laplacian eigenfunctions.
It has been recently established that the local high-frequency behavior of these monochromatic random waves matches that of the Euclidean Berry model at least in a local sense via a pullback procedure \cite{CanzaniHanin2020,DierickxNourdinPeccatiRossi2023,Keeler,STinprogress}. 
In the context of monochromatic RWs and compact setting, it is important to mention two special cases. The first is the sphere $\mathcal{M}=\mathbb{S}^n$, 
where random hyperspherical harmonics find extensive applications in cosmology, particularly in modeling the Cosmic Microwave Background (CMB) radiation \cite{MarinucciPeccati-Sphere}. The second is the flat torus $\mathcal{M}=\mathbb{T}^n$, where the corresponding fields are known as arithmetic random waves (ARW).
The literature is vast; we refer for instance to  \cite{WigmanAnnMath,RudWig08,MPRW,cammarota2019nodal} for the study of geometric functionals of ARW and to \cite{W,Marinucci2020, marinucci2023laguerre,Marinucci2015,CammarotaM2018} for that of spherical harmonics, finally to \cite{Marinucci2023,igorsurvey} for some surveys on this subject. 
Random Gaussian generalized eigenfunctions for the Laplace-Beltrami operator have been considered on non-compact Riemannian manifolds as well. Recently,  \cite{GrottoPeccati2023} constructed such waves on the $n$-dimensional hyperbolic space. Their construction relies on the Helgason-Fourier transform on the hyperbolic space, which is a symmetric space. The group Fourier transform is available on the Heisenberg group, but it 
leads to an operator-valued map because of the presence of infinite-dimensional unitary irreducible representations. To bypass this difficulty, we rely on harmonic analysis on $\Hei$, where deterministic eigenfunctions are known \cite{Thangavelu98}, leaving the  analysis of the group Fourier transform approach to future works.

\subsection{Main results}\label{sec:main}
The aim of this paper is to construct and characterize a complex-valued smooth centered left invariant Gaussian random field consisting of generalized eigenfunctions for the sub-Laplacian $\mathcal{L}$, the Berry-Heisenberg random wave. In analogy to the Euclidean case, where the covariance of the Berry random wave model satisfies \eqref{eqn.intro.idk} and is invariant by rotations, the covariance function $C$ of the Berry-Heisenberg field should satisfy 
\begin{align}\label{eqn.intro.idk2}
    &\mathcal{L} C = - \alpha \, C, & C(e)=2,
\end{align}
where $e\in \Hei$ denotes the identity element, and $C$ should be a function invariant under the action of the subgroup of isometries fixing the identity. Therefore, the covariance function admits the following form
\begin{align}
    &C(g)=2 F_{\alpha}^{(c)} (g) :=2c_{-1}  J_{0} ( \sqrt{\alpha} \,|z| ) + 2\sum_{k=0}^{\infty} c_{k} \, \varphi_{\alpha}^{k} (g),\label{eq:intro.e.cov.H.c}
\end{align}
where $\varphi^{k}_{\alpha}$ are radial deterministic eigenfunctions for $\mathcal{L}$ which will be defined in \eqref{e.sum.eps} and $(c_{k})_{k=-1}^{\infty}$ are  some non-negative coefficients with $\sum_{k=-1}^{\infty} c_{k} =1$ specified in \eqref{def:O}. Our main results include Theorem \ref{thm.smooth.field} and Theorem \ref{thm.stat.field.H}. In the former we prove that if a Gaussian field has covariance function given by $F^{(c)}_{\alpha}$, then its trajectories are almost surely smooth, and in the latter we show the existence of a left-invariant complex-valued centered Gaussian field  $\{  \xi_{\alpha,c} (g), \, g\in \Hei \}$ with covariance function given by $C(g)=2 F_{\alpha}^{(c)} (g)$. Moreover, we prove that the field is unique in the sense that if  $\{  \zeta  (g), \, g\in \Hei \}$ is a  left-invariant complex-valued centered Gaussian field, whose covariance is invariant under isometries fixing $e\in \Hei$ and such that $\mathcal{L} \zeta = -\alpha \, \zeta$, then $\zeta =  \xi_{\alpha,c}$ in distribution. As a byproduct of Theorem \ref{thm.stat.field.H}, we also have the following representation in terms of deterministic eigenfunctions 
\begin{equation}\label{representation.r.f.H.intro}
    \xi_{\alpha,c} (g) =  \sqrt{c_{-1}}\, b^{\C}_{\alpha} (z)+\sum_{k=0}^{\infty} \sqrt{c_{k}} \, \xi^{k}_{\alpha} (g) , \quad g=(z,t) \in \Hei,
\end{equation}
where $\{ b^{\C}_{\alpha} (z) , z \in \R^{2} \}$ is a complex-valued Berry random field on $\R^{2}$ and $\xi^{k}_{\alpha} (g)$ is the $k-$component of $\xi_{\alpha,c}$ defined in \eqref{e.field.component}. We refer to the field  $\{  \xi_{\alpha,c} (g), \, g\in \Hei \}$ as the  $(\alpha, c)$-Berry-Heisenberg random wave.

In Proposition \ref{prop.berry.horizontal}  and  Proposition \ref{prop.approx.CLT} we establish some connections  to the Euclidean Berry random wave model. More precisely, in Proposition \ref{prop.berry.horizontal} we show that, for a  fixed $\alpha>0$ and $k\in \mathbb{N}$, if $\xi^{k}_{\alpha} (g)$ denotes the $k$-component of $\xi_{\alpha,c}$, that is, the field with covariance $2\varphi_{\alpha}^{k}(g)$, then, as $k\rightarrow \infty$, the covariance function of $\xi_{\alpha}^{k}$ converges to the covariance function of the field $\Xi(g) := b^{\C}_{\alpha} (z)$ for $g=(z,t) \in \Hei$, which can be viewed as a  complex-valued Berry random field on the horizontal sub-bundle of $\Hei$. 
We leave the investigation of whether a similar property holds on more general sub-Riemannian manifolds as a direction for future research. In Proposition \ref{prop.approx.CLT} we show that the field $\{\xi_{\alpha, c}(g), \, g\in \Hei\}$  can be represented as the limit of superpositions of independent waves, where now waves are to be interpreted as eigenfunctions of the sub-Laplacian $\mathcal{L}$, in analogy with \eqref{eq:supBerry} for Berry random waves. 

Gaussian random functions on the complex plane with covariance structure similar to \eqref{eq:intro.e.cov.H.c} have recently been studied in \cite{HaimiKolianderRomero22,AbreuPereiraRomero17}.

The construction of Berry-Heisenberg random fields is based on a result of Yaglom \cite{Yaglom60}, where he proved a representation formula for positive-definite functions on a locally compact topological group $G$ of type I in terms of unitary irreducible representations of $G$ and positive operator-valued measures. This was then used to obtain a decomposition of left-invariant, i.e. left-stationary, random fields on  $G$. For completeness, we present a proof of Yaglom's original result in Appendix \ref{appendix.A}. Our proof differs from the original one as it relies on direct integrals of unitary representations in terms of irreducible ones, see Theorem \ref{thm.Bochner}. Note that, in the case of compact Lie groups, Yaglom's decomposition of random fields on groups can be viewed as a random analog of the classical Peter-Weyl theorem in representation theory, see also \cite{BaldiRossi2014}. Unitary irreducible representations of  the Heisenberg group are known  thanks to the Stone-von Neumann Theorem and they correspond to either characters (one-dimensional) or Schr\"odinger representations (infinite-dimensional). Our approach combines such representations with Yaglom's decomposition to construct left-invariant random fields whose trajectories are eigenfunctions for the sub-Laplacian. Moreover, from the proof of Theorem  \ref{thm.stat.field.H} it is clear how such representations affect the covariance structure of the field in (\ref{eq:intro.e.cov.H.c}). Indeed, Schr\"odinger representations  give rise to the family of radial functions $ \varphi^{\alpha}_{k}$, while characters contribute with  $J_{0} (\sqrt{\alpha} |z|)$. We remark that representation theory for higher dimensional Heisenberg groups and more general nilpotent Lie groups is known, see \cite{Kirillov1960, Kirillov1963, Kirillov2004, AsaadGordina2016}, and we will address construction of Gaussian random eigenfunctions in this setting in the future. 

\subsection{Comparison with other fields}\label{sec:comparison}
The classical Berry random field $\{ b^{\C}_{\alpha} (x) , \, x\in \R^{n}\}$ is the unique (in distribution) isotropic centered Gaussian random field with variance two such that its trajectories are almost surely generalized eigenfunctions for the Laplace operator with eigenvalue $\alpha>0$. Geometrically, this means that the distribution of the field is invariant under rotations and hence its covariance is the unique smooth radial solution to \eqref{eqn.intro.idk}. On the Heisenberg group, radial functions are the ones invariant under the action of the subgroup of isometries fixing the identity $e\in \Hei$, which play the same role as rotations in $\R^{n}$. Contrary to the Euclidean case, for a fixed $\alpha>0$ there are countably many radial solutions to the equation 
    \[
    \mathcal{L}f =- \alpha \,f , \quad f(e)=2,
    \]
which we denote by $g\rightarrow 2\varphi_{\alpha}^{k}(g)$, $k\in \mathbb{N}$, and $g\rightarrow 2 J_{0} (\sqrt{\alpha} |z|)$, for $g=(z,t)$, where $J_{0}$ is the Bessel function of first kind and order $0$, and $\varphi_{\alpha}^{k}$ are some special functions which are defined using Schr\"odinger representations of $\Hei$ and can be explicitly computed in terms of Laguerre polynomials, see Proposition \ref{prop.properties.e.functions}. Note that $2J_{0} (\sqrt{\alpha} |z|)$ is the covariance function of  a Berry random wave on $\R^{2}$.

    Note that the expansion \eqref{representation.r.f.H.intro} exhibits a structural analogy with the classical representation of isotropic spherical random fields, see e.g. \cite{MarinucciPeccati-Sphere}. If
    \begin{align*}
        T : \mathbb{S}^{2} \times \Omega \longrightarrow \C
    \end{align*}
     is a centered, isotropic Gaussian field on the unit sphere $\mathbb{S}^{2}$, then \begin{align*}T(x) = \sum_{\ell=0}^{\infty} \sum_{m=-\ell}^{\ell} a_{\ell m} Y_{\ell m}(x), \quad x \in \mathbb{S}^2,\end{align*}where $\{Y_{\ell m}, \, m=\ell, \dots, \ell , \, \ell\in \mathbb{N}\}$ are the standard spherical harmonics, which form an orthonormal basis of $L^2(\mathbb{S}^2)$ consisting of Laplacian eigenfunctions. If one views $\mathbb{S}^2$ as a homogeneous space of $SO(3)$, then the above decomposition is a direct  consequence of the Peter-Weyl Theorem for compact topological groups. Moreover, on $\mathbb{S}^2$, the random coefficients can be defined by the relation
    \begin{align*}
        & a_{\ell m} := \int_{\mathbb{S}^{2}} T(x) \overline{Y_{\ell m}} (x) dx,
    \end{align*}
    which is well-defined since $Y_{\ell m}$ is in $L^{2} (\mathbb{S}^{2})$. Note that generalized eigenfunctions for the sub-Laplacian on $\Hei$ are not in $L^2(\mathbb{H})$ since $\Hei$ is non-compact, and the Peter-Weyl Theorem is not available in this setting, but a representation theorem for left-invariant fields on locally compact groups of type I is still available, see Theorem \ref{thm.stationarygroup}.   On the other hand, a closer parallel can be drawn with the standard expansion of the complex-valued planar Berry random wave $b_{\alpha}(z)$, $z\in \R^{2}$. In polar coordinates $z=(r,\theta)$, this field can be represented as, see \cite[Theorem 5.7.3]{AdlerTaylor07} and \cite[Eq. 1.5]{NourdinPeccatiRossi19} \begin{equation}\label{e.rep.berry.complex}b_{\alpha}(z) =b_\alpha(r,\theta)= \sum_{m=-\infty}^{\infty} a_{m} J_{|m|} (\sqrt{\alpha}r) e^{im\theta},\end{equation}where $\{a_{m}\}$ is a family of i.i.d. complex Gaussian random variables. The representation \eqref{representation.r.f.H.intro} can be interpreted as an analog to \eqref{e.rep.berry.complex} to the Heisenberg group.\\

The paper is structured as follows. In Section \ref{sec:background}, we collect the necessary background material to understand the framework and the tools required for the main theorems. More precisely, Bochner Theorem for the representation of stationary fields on $\mathbb{R}^n$, and more properties on Berry's model. We also collect general facts about representation theory and  random fields on locally compact groups of type I. We conclude Section \ref{sec:background} with the Heisenberg group and harmonic analysis on it. Statements and proofs of the main results can be found in Section \ref{sec:berry-heisenbergRW}. Finally, Appendix \ref{appendix.A} collects the proofs of classical theorems concerning representation of positive-definite functions and random fields on locally compact groups of type I. The results presented in this section are classical, but we included proofs for completeness as they differ from the existing ones in the literature, being based on direct integrals of Hilbert spaces and representations.

\section{Preliminaries}\label{sec:background}

In this section we recall the necessary background: stationary fields on $\R^{n}$ and on a general locally compact group of type I, and harmonic analysis on the Heisenberg group. Our main references are \cite{AdlerTaylor07, Thangavelu98, Yaglom60}.

\subsection{Stationary fields on $\R^{n}$}\label{sec:stationary_fields}

A complex-valued random field $f=\{f(t) , \, t\in T\}$ is a collection of complex-valued random variables $f(t)$ defined on a common probability space $\left( \Omega, \mathcal{F}, \Prob \right)$, where $T\subseteq \R^{n}$ is a subset. The mean and covariance functions are defined as
\begin{align*}
    & m(t) := \E \left[ f(t) \right], & C(t,s) := \E \left[ ( f(t) - m(t))( \overline{f(s)} -  \overline{m(s)})\right],
\end{align*}
and the field is called \textit{mean-square continuous} if 
\begin{align*}
    \lim_{t\rightarrow t_{0}} \E \left[ \vert f(t) -f(t_{0})\vert^{2} \right]=0,
\end{align*}
for every $t_{0} \in \R^{n}$. A field  $f$ is said to be \textit{strictly stationary} (or  \textit{strictly homogeneous}) if for any $k\geq 1$ and any set of points $t_{1}, \ldots , t_{k} \in \R^{n}$ one has that, for every $\tau \in \R^{n}$
\begin{equation}\label{e.real.stationary}
    \left( f(t_{1}) , \ldots , f(t_{k}) \right) \stackrel{(d)}{=}  \left( f(t_{1} + \tau) , \ldots , f(t_{k}+\tau) \right),
\end{equation}
where by $\stackrel{(d)}{=}$ we denoted the equality in law. Note that for a strictly stationary field the mean function is constant and the covariance function only depends on the difference $t-s$. If the field $f$ has finite variance, its mean function is constant, and its covariance function only depends on $t-s$, then $f$ is said to be  \textit{weakly stationary} or  \textit{second-order stationary}, or simply  \textit{stationary}. 

Let us now recall the classical spectral representation theorem for stationary random fields on $\R^{n}$. We assume here that $f$ is stationary, $m(t) \equiv 0$, and with an abuse of notation we write
\begin{align*}
    C(t-s) = C(t,s)=  \E \left[  f(t)  \overline{f(s)}\right].
\end{align*}
By Bochner theorem \cite{Bochner33} and \cite[Theorem 5.4.1]{AdlerTaylor07}, a continuous function $C:\R^{n} \rightarrow \mathbb{C}$ is non-negative definite, i.e. a covariance function, if and only if there exists a finite measure $\nu$ on $\R^{n}$ such that 
\begin{align}\label{eqn.cov.Rn}
        C(t) = \int_{\R^{n}} e^{i  t \cdot \lambda} d \nu \left( \lambda\right),
\end{align}
for all $t\in \R^{n}$. Let us recall that a complex $\nu$-noise is a $\mathbb{C}$-valued, set-indexed process $W$ satisfying 
\begin{align*}
    & \E [ W(A)] =0 , \quad \E[W(A) \overline{W(A) }] = \nu (A),
    \\
    & A\cap B = \emptyset \implies W(A\cup B)= W(A) + W(B) \, a.s.
    \\
    & A\cap B = \emptyset \implies  \E[W(A) \overline{W(B) }] =0. 
\end{align*}

\begin{thm}\label{thm.stationary.Rn} \cite[Theorem 5.4.2]{AdlerTaylor07}
    Let $f=\{ f(t), \, t\in \R^{n}\}$ be a mean-square continuous, centered (Gaussian) stationary random fields on $\R^{n}$ with covariance function \eqref{eqn.cov.Rn}. Then there exists a complex (Gaussian) $\nu$-noise $W$ on $\R^{n}$ such that 
    \begin{align}\label{e.spectral.rep.real}
        f(t) = \int_{\R^{n}} e^{i  t\cdot \lambda}  dW(\lambda),
    \end{align}
    where the equality holds in mean-square for each $t\in \R^{n}$.

    Conversely, the random field $f$ defined by \eqref{e.spectral.rep.real} has covariance function given by 
    \begin{align*}
        C(t,s) = \int_{\R^{n}} e^{i ( t-s) \cdot \lambda } d \nu (\lambda),
    \end{align*}
    and so is stationary.
\end{thm}
One refers to $\nu$ and $W$ as the spectral measure and the spectral process corresponding to $f$, respectively.   We recall that a stationary field is said to be \textit{isotropic} if the covariance function $C$ satisfies 
\begin{align*}
    C(t) = C(|t|).
\end{align*}
For centered, mean-square continuous, isotropic random fields on $\R^{n}$, the spectral representation \ref{e.spectral.rep.real} takes the following form, see \cite[Theorem 5.7.3]{AdlerTaylor07} 
\begin{align}\label{e.isotropic.rep.real}
    f(t)= f(r,\theta) = \sum_{m=0}^{\infty} \sum_{\ell=1}^{d_{m}} f_{ m\ell } (r) h^{(n-1)}_{m \ell } (\theta),
\end{align}
where $\{f_{m \ell }, \, m, \ell\}$ is a family of mutually uncorrelated, stationary, one-dimensional processes, $\{h^{(n-1)}_{m \ell }\}$ are the spherical harmonics on $\mathbb{S}^{n-1}$ and $d_m$ is the dimension of the space of spherical harmonics of degree $m$ in dimension $n$.

\subsubsection{Berry random wave model}\label{sec:berry_detail}
Let us consider the Berry random wave model as described in the introduction. The associated real-valued model on $\mathbb{R}^n$ is defined as the isotropic Gaussian random field with covariance function
\begin{align}\label{e.berry.cov}
    \E \left[ b_{\alpha} (x) b_{\alpha}(y) \right]= j_{n} \left( \sqrt{\alpha} \vert x- y \vert \right), \; \; x,y,\in \R^{n},
\end{align}
where $\alpha>0$ is fixed and 
\begin{align*}
    j_{n} (x) = j_{n} (|x|) = \frac{1}{\omega_{n}} \int_{\mathbb{S}^{n-1}} e^{i x\cdot \theta} d\sigma_{n-1} (\theta) = \frac{\nu ! 2^{\nu}}{|x|^{\nu}}J_{\nu} (\vert x \vert ) , \quad \nu = \frac{n}{2}-1,
\end{align*}
with $\omega_n$ being the measure of the $n$-dimensional unitary ball and $J_{\nu}$ denotes the Bessel function of first kind and order $\nu$.  
The covariance function in (\ref{e.berry.cov}) can be interpreted in terms of the Fourier transform on $\R^{n}$. More precisely, the Fourier inversion formula can be written as \cite[Section 2.3]{Thangavelu98}
\begin{align*}
    f(x) =\frac{1}{(2\pi)^{\frac{n}{2}}} \int_{0}^{\infty} \left(Q_{\alpha}f \right) (x) \alpha^{n-1} d\alpha = \frac{1}{2(2\pi)^{\frac{n}{2}}}\int_{0}^{\infty} \left(Q_{\sqrt{\alpha}}f \right) (x) \alpha^{\frac{n}{2}-1} d\alpha,
\end{align*}
where the operator $Q_{\sqrt{\alpha}}$ is given as a convolution 
\begin{align}\label{e.convolution.R}
    \left( Q_{\sqrt{\alpha}} f\right) (x) = \left( f \ast \varphi_{\sqrt{\alpha}}  \right) (x) = \int_{\R^{n}} f(y) \varphi_{\sqrt{\alpha}} (x-y) dy,
\end{align}
and 
\begin{align*}
    \varphi_{\sqrt{\alpha}} (x) = \frac{1}{\left( \sqrt{\alpha} |x| \right)^{\nu}} J_{\nu} \left( \sqrt{\alpha} \, |x| \right) = \frac{1}{\nu ! 2^{\nu}} j_{n} \left( \sqrt{\alpha} \, |x| \right), \quad \nu = \frac{n}{2}-1.
\end{align*}
Note that  $Q_{\sqrt{\alpha}}\ast f$ is an eigenfunction for the Euclidean Laplacian with eigenvalue $\alpha$. We can then rewrite \eqref{e.berry.cov} as
\begin{align}\label{e.Berry.R2}
      \E \left[ b_{\alpha} (x) b_{\alpha}(y) \right]=\nu ! \, 2^{\nu}\varphi_{\sqrt{\alpha}} (x-y),  \; \; x,y \in \R^{n}.
\end{align}

The representation given in \eqref{eq:supBerry} becomes, in the real-valued case,
\begin{align}\label{eqn.berry.approx.waves}
    \frac{1}{\sqrt{N}} \sum_{i=1}^{N} \cos \left( \sqrt{\alpha} \, x \cdot y_{i} + \phi_{i} \right) , \quad x \in \R^{n}, \quad N \rightarrow \infty ,
\end{align}
where $y_{i}$'s and $\phi_{i}$'s are i.i.d. uniform random variables on $\mathbb{S}^{n-1}$ and $[0,\pi]$, respectively. An analogue to \eqref{eqn.berry.approx.waves} for Berry-Heisenberg random waves will be proven in Proposition \ref{prop.approx.CLT}. Moreover,  if $n=2$  then complex Berry random field  can be represented as 
\begin{align}
    & b^{\C}_{\alpha} (x) = b^{\C}_{\alpha} (r, \theta) = \sum_{m=-\infty}^{\infty} a_{m} J_{|m|} (\sqrt{\alpha} \, r) e^{im\theta}, \quad x= (r,\theta), \label{eqn.berry.complex}
    \\
    & \E \left[ b^{\C}_{\alpha} (x) \overline{b^{\C}_{\alpha} (y)} \right] = 2J_{0} \left( \sqrt{\alpha} \vert x-y \vert \right), \label{eq.complex.cov.berry}
\end{align}
where $a_{m}$ are i.i.d. complex Gaussian random variables such that $\E[a_{m}]=0$, $\E[|a_{m} |^{2} ]=2$.  
A similar decomposition for Berry-Heisenberg random waves will be proven in Theorem \ref{thm.stat.field.H}.

\subsection{Stationary fields on locally compact groups}\label{sec:stationary on group} Our goal is to state an analog of Theorem \ref{thm.stationary.Rn} for stationary random fields on locally compact groups of type I, following \cite{Yaglom60, Folland16}.  Let us start by recalling some general facts about representation theory of locally compact Lie groups. Let $G$ be a topological group, that is, a topological space endowed with a group structure such that the following maps are continuous 
\begin{align*}
    & L_{g} : G \longrightarrow G, & R_{g} : G \longrightarrow G,
    \\
    & L_{g}(h)= g^{-1}h, & R_{g}(h) = hg,
\end{align*}
for every fixed  $g\in G$, and set $ I_{g}(h) :=g^{-1}hg=L_{g} R_{g} (h)$. The maps $L_{g}, R_{g}$, and $I_{g}$ are the left and right translation by $g$, and the inner automorphism respectively. A unitary representation of a locally compact group $G$ is a group homomorphism 
\begin{align*}
    \pi: G \longrightarrow U(H_{\pi}),
\end{align*}
of the group $G$ into the group of unitary operators $ U(H_{\pi})$ on some Hilbert space $H_{\pi}$ which is continuous with respect to the strong operator topology, that is, 
\begin{align*}
    & \pi (gh) = \pi(g) \pi(h) , & \pi(g^{-1}) = \pi(g)^{-1} = \pi(g)^{\ast},
\end{align*}
where $\pi(g)^{\ast}$ denotes the adjoint of the operator $\pi(g)$,  and the map 
\[
g\longrightarrow \pi (g)u 
\]
is continuous from $G$ to $H_{\pi}$ for every fixed $u\in H_{\pi}$. The underlined Hilbert space $H_{\pi}$ is called the representation space, and the dimension of the representation $\pi$ is defined to be $\dim \pi := \dim H_{\pi}$. One can define an equivalence relation on the set of unitary representations in the following way. We say that two representations $\pi_{1}$ and $\pi_{2}$ are equivalent if there exists an intertwining operator, that is, there exists an operator $T: H_{\pi_{1}} \longrightarrow H_{\pi_{2}}$  such that the following diagram is commutative 
\begin{center}
\begin{tikzcd}
    H_{\pi_{1}} \arrow[r, "\pi_{1}(g)"] \arrow[d, "T"'] & H_{\pi_{1}} \arrow[d, "T"] \\
    H_{\pi_{2}} \arrow[r, "\pi_{2}(g)"] & H_{\pi_{2}}
\end{tikzcd}
\end{center}
for all $g\in G$. Given a unitary representation $\pi$, a closed subspace $M$ of $H_{\pi}$ is said to be an invariant subspace for $\pi$ if $\pi(g) M \subseteq M$ for all $g\in G$. A representation $\pi$ is irreducible if it does not admit non-trivial invariant subspaces. For more details on unitary representations, equivalence classes, and irreducibility, we refer to \cite[Chapter 3]{Folland16}. The set of equivalence classes of unitary representations of $G$, denoted by $\widehat{G}$, is the unitary dual of $G$. The study of random fields on locally compact groups is strongly connected to the theory of unitary representations. The latter is a complicated topic on a general non-compact non-Abelian group, but the theory is well-developed for type I groups. These are groups for which the unitary dual $\widehat{G}$ is a quasi-regular topological space and every representation can be decomposed into the direct sum of multiplicity free representations which decompose into a direct integral of irreducible representations, see Appendix \ref{appendix.A} for more details. A function $B: G \longrightarrow \mathbb{C}$ is said to be positive-definite if the matrix 
\begin{align*}
    \left( B(g_{j}^{-1} g_{i}) \right)_{i,j=1}^{n}
\end{align*}
is positive-definite for every choice of $g_{1}, \ldots , g_{n} \in G$ and every $n\in \mathbb{N}$. One has the following characterization of positive-definite functions on locally compact groups of type I.

\begin{thm}\label{thm.Bochner}
    Let $G$ be a separable, locally compact group of type I. A function $B: G \longrightarrow \mathbb{C}$ is positive definite if and only if it can be represented as 
    \begin{equation}\label{eqn.pos.def.fun}
        B(g) = \operatorname{Tr} \left( \pi (g)  F(\widehat{G} ) \right) ,
    \end{equation}
    where $\widehat{G}$ denotes the unitary dual of $G$,  $\pi$ is a unitary representation of $G$ acting on the Hilbert space $H_{\pi}$, and $F$ is a positive operator-valued measure over $\widehat{G}$, whose values are  Hermitian non-negative definite operators on the representation space $H_{\pi}$, such that 
    \begin{align}\label{e.condition.pos.def}
       \operatorname{Tr} \left( F(\widehat{G}) \right) < \infty.
    \end{align}
\end{thm}
Theorem \ref{thm.Bochner} was originally proven in \cite{GelfandRaikov43, Naimark59} in the form of Theorem \ref{thm.Bochner1}, and a different version can be found in \cite[Theorem 3]{Yaglom60} and \cite[Theorem 2.18]{Malyarenko2013}. In Appendix \ref{appendix.A}  we present an alternative proof of Theorem \ref{thm.Bochner} by means of direct integrals of unitary irreducible representations and Hilbert spaces.

Let us now describe the random fields corresponding to the covariance function \eqref{eqn.pos.def.fun}. Let $\left(\Omega, \mathcal{F}, \Prob\right)$ be a probability space and $L^{2}\left(\Omega, \Prob\right)$ be the Hilbert space of complex-valued finite variance random variable. A random field $\xi$ on $G$ is a function 
\begin{align*}
    \xi : G \longrightarrow L^{2}\left(\Omega, \Prob\right),
\end{align*}
or equivalently, a collection of complex-valued finite-variance random variables $\{ \xi(g), \; g\in G\}$ defined on a common probability space $\left(\Omega, \mathcal{F}, \Prob\right)$. We assume that the field $\xi$ is mean-square continuous, that is, 
\begin{align*}
    & \E \left[ \left| \xi(g_{n} )- \xi(g)  \right|^{2} \right] \longrightarrow 0 , \text{ if } g_{n} \longrightarrow g \in G.
\end{align*}
A field is called \textit{left-homogeneous} or \textit{left-stationary} if $\E \left[\xi(g) \right]= \E \left[\xi(L_{h}(g)) \right]$ and  $\E \left[|\xi(g)|^{2} \right]= \E \left[|\xi(L_{h}(g))|^{2} \right]$ for all $h\in G$.  It is called right-homogeneous (resp. two-way homogeneous) if $L_{g}$ is replaced by $R_{g}$ (resp. $I_{g}$). For a detailed treatment of Abelian groups and compact topological groups, we refer to \cite[Section 2]{Yaglom60} and \cite{Malyarenko2013}. Before we state an analog of Theorem \ref{thm.stationary.Rn} for locally compact groups of type I  \cite[Theorem 4]{Yaglom60}, let us recall the notion of random linear operators. This is a measurable function $Z$ on a measurable space $(B, \mathcal{B})$ defined by
\begin{align*}
    Z: \Omega \times \mathcal{B} \longrightarrow \mathcal{L} (H),
\end{align*}
where $ \mathcal{L} (H)$ denotes the space of linear operators on some Hilbert space $H$, such that the map 
\begin{align*}
    &\mathcal{B}  \longrightarrow \R,
    \\
    & A\rightarrow \langle Z(A) v,w\rangle,
\end{align*}
is a random measure for every $v,w\in H$. We have the following representation theorem. 

\begin{thm}\label{thm.stationarygroup}\cite[Theorem 4]{Yaglom60}
    Let $\xi = \{ \xi(g), \; g\in G\}$ be a random field on a separable locally compact group $G$ of type I. Then $\xi$ is left-homogeneous if and only if 
    \begin{equation}\label{e.stationary.r.f.group}
        \xi(g) = \operatorname{Tr} \left( \pi(g)Z(\widehat{G}) \right) ,
    \end{equation}
    where $\widehat{G}$ denotes the unitary dual of $G$, $\pi$ is a unitary representation of $G$ acting on a Hilbert space $H_{\pi}$, and $Z$ is a random linear operator over $\widehat{G}$ such that 
    \begin{align}\label{eqn.cov.random.op.measure}
        \E \left[ \langle Z(A) u_{1} , u_{2} \rangle_{H_{\pi}} \overline{\langle Z(B) v_{1} , v_{2} \rangle}_{H_{\pi}}  \right]= \langle F(A \cap B) u_{1} , v_{1} \rangle_{H_{\pi}}  \langle u_{2} , v_{2} \rangle_{H_{\pi}} ,
    \end{align}
for any measurable subsets $A, B$ of $\widehat{G}$ and any vectors $u_{1}, u_{2}, v_{1}, v_{2} \in H_{\pi}$, where $F$ is an operator-valued measure on $\widehat{G}$, whose values are Hermitian non-negative definite operators on $H_{\pi}$ satisfying \eqref{e.condition.pos.def}. Moreover, the covariance function of $\xi$ is given by \eqref{eqn.pos.def.fun}.
\end{thm}
In Appendix \ref{appendix.A} we express \eqref{eqn.pos.def.fun}, \eqref{e.condition.pos.def}, \eqref{e.stationary.r.f.group}, and \eqref{eqn.cov.random.op.measure} explicitly in terms of a basis for the representation space $H_{\pi}$, by means of direct integrals of unitary representations.

\begin{rmk}
    Note that if $G=\R^{n}$ then $\widehat{G} =\R^{n}$, and for each $\lambda\in \R^{n}$ a representative in the corresponding class of unitary irreducible representations is given by 
    \begin{align*}
       & \pi_{\lambda} : \R^{n} \longrightarrow U(\C), & \pi_{\lambda} (x) : \C \longrightarrow \C , & & \pi_{\lambda} (x)z= e^{i \lambda \cdot x} z,
    \end{align*}
    and \eqref{e.stationary.r.f.group} reduces to \eqref{e.spectral.rep.real}.
\end{rmk}

As pointed out in the next section, in the case of the Heisenberg group, unitary irreducible representations are either one-dimensional, i.e. characters acting on $\C$, or infinite-dimensional Schr\"odinger representations acting on $L^{2}(\R)$, and \eqref{eqn.pos.def.fun} and \eqref{e.condition.pos.def}  take a simpler form.

\subsection{Heisenberg group} \label{sec:heisenberg}
Here we recall some general properties of the Heisenberg group and harmonic analysis on it. 

\subsubsection{General facts}\label{sec:heisenberg_general}The Heisenberg group $\Hei$ is the set $\R^{3}$ endowed with the following non-Abelian group operation 
\begin{align*}
    g_{1}g_{2} = (x_{1}, y_{1} , t_{1} ) (x_{2}, y_{2} , t_{2} )= \left(x_{1} + x_{2}, y_{1} + y_{2} , t_{1} + t_{2} + \frac{1}{2} ( x_{1} y_{2} - x_{2} y_{1}) \right).
\end{align*}
The left-invariant vector fields at $g= (x,y,t) \in \Hei$ are given by 
\begin{align*}
    & X_{g} = \partial_{x} -\frac{1}{2} y \partial_{t}, & Y_{g} = \partial_{y} +\frac{1}{2} x \partial_{t}, & & T_{g} = \partial_{t},
\end{align*}
and the set $\{X,Y\}$ satisfies H\"ormander's condition since $[X,Y]=T$. This ensures that the sub-Laplacian 
\begin{equation}\label{e.sub.lap}
  \mathcal{L} := X^{2} + Y^{2}
\end{equation}
is a hypoelliptic operator. For $g\in \Hei$, let us consider an inner product on $\mathcal{H}_{g}:=\text{span}\{ X_{g},Y_{g}\}$ so that $\{X_{g}, Y_{g}\}$ is an orthonormal system. This describes a sub-Riemannian structure on $\Hei$ where the horizontal distribution is given by  $\mathcal{H}:=\text{span}\{ X,Y\}$. For $\lambda\in (0,\infty)$, we define the dilation by $\lambda$
\begin{align*}
    & \delta_{\lambda} : \Hei \rightarrow \Hei , & \delta_{\lambda} (x,y,t) = (\lambda x, \lambda y, \lambda^{2} t),
\end{align*}
and note that the sub-Laplacian $\mathcal{L}$ is homogeneous of degree $2$ with respect to $\delta_{\lambda}$, that is,
\begin{align*}
    \mathcal{L} f (\delta_{\lambda})= \lambda^{2} ( \mathcal{L} f) (\delta_{\lambda}).
\end{align*}
We recall that a curve $s\rightarrow \gamma(s)\in \Hei$ is said to be horizontal if $\gamma^{\prime}(s) \in \mathcal{H}_{\gamma(s)}$ for a.e. $s\in\R$. The Heisenberg group comes with a natural distance which we define now.
\begin{definition}
    For any $g_{1}, g_{2} \in \Hei$, the Carnot-Carath\'eodory distance is defined by 
    \begin{align*}
         d_{cc} (g_{1}  , g_{2} ) := \inf \left\{ \int_{0}^{1} \Vert \gamma^{\prime}(t) \Vert_{\gamma (s)} ,  \gamma: [0,1] \longrightarrow \Hei , \gamma(0) = g_{1} , \gamma(1) = g_{2} ,  \gamma \text{ is horizontal }\right\}
    \end{align*}
    where $\Vert \cdot \Vert_{g}$ denotes the norm in $\mathcal{H}_{g}$ induced by the sub-Riemannian structure. 
\end{definition}
In addition to the Carnot-Carath\'eodory distance, we will use the following homogeneous norm on $\Hei$
\begin{align*}
   & | g| := \sqrt{|z|_{\R^{2}} + |t|} , & g= (z,t)= (x,y,t) \in \Hei,
\end{align*}
which induces a distance equivalent to the Carnot-Carath\'eodory one, that is, 
\[
c^{-1}\, | h^{-1} g | \leqslant d_{cc} (g,h) \leqslant c  | h^{-1} g |,
\]
for some constant $c>1$ independent of $g,h\in \Hei$, see \cite[Chapter 5]{BonfiglioliLanconelliUguzzoni2007}.

The unitary irreducible representations of the Heisenberg group are well known. For $\beta \in \R^{2}$ and $g=(z,t)\in \Hei$, let us set 
\begin{align*}
    & \pi_{\beta} (g): \C \longrightarrow \C , & \pi_{\beta} (g)w= e^{2\pi i \, \beta \cdot z } w,
\end{align*}
which is a unitary irreducible representation on the Hilbert space $H_{\pi_{\beta}}= \C$ and corresponds to characters, that is, one-dimensional representations. For $\lambda\neq 0$, $g=(x,y,t) \in \Hei$, and $\phi \in L^{2} (\R, d\xi)$, let us consider the function 
\begin{align*}
    \pi_{\lambda}(g) \phi (\xi) := e^{i\lambda \left( x\xi + \frac{1}{2} xy + t \right)} \phi ( y+\xi).
\end{align*}
It is easy to see that $\pi_{\lambda}$ is a unitary representation on the Hilbert space $H_{\pi_{\lambda}}= L^{2} (\R, d\xi)$ of complex-valued square integrable functions. Note that, for the Heisenberg group, the representation space $H_{\pi_{\lambda}}$ is independent of the representation. The unitary dual of the Heisenberg group is well-understood by the Stone-von Neumann theorem \cite[Theorem 6.50]{Folland16}.
\begin{thm}[Stone-von Neumann]\label{thm.stone.neuman}
Every unitary irreducible representation of $\Hei$ is equivalent to one and only one of the following 
\begin{align*}
    & \pi_{\lambda}, & \lambda \neq 0 , 
    \\
    &   \pi_{\beta} , & \beta \in \R^{2}.
\end{align*}
\end{thm}

In particular, it follows that the unitary dual can be identified with the space 
\begin{align*}
    & \widehat{\Hei}= \R^{2} \,  \dot{\cup} \, \R^{\ast},
\end{align*}
where $\R^{\ast}$ denotes the set of non-zero real numbers, proving that $\Hei$ is of type I and all irreducible unitary representations of $\Hei$ are either infinite-dimensional, Schr\"odinger representations $\pi_{\lambda}$, or one-dimensional, characters $\pi_{\beta}$. The representations  $\pi_{\lambda}$'s can be used to construct the group Fourier transform. For $\lambda\in \R^{\ast}$ and $f\in L^{2} ( \Hei, dg)$, where $dg$ denotes the Haar measure, let us set
\begin{align*}
    & \hat{f} (\lambda) : L^{2}(\R, d\xi) \longrightarrow L^{2}(\R, d\xi),
    \\
    &  \hat{f} (\lambda) \phi:= \int_{\Hei} f(g) \pi_{\lambda} (g) \phi dg,
\end{align*}
where the integral is a Bochner integral taking values in $ L^{2}(\R, d\xi)$. It can be shown that the operator $\hat{f} (\lambda)$ is Hilbert-Schimdt for  $f\in L^{2} ( \Hei, dg)$, \cite[Section 1.3]{Thangavelu98}.

\subsubsection{Fourier analysis on $\mathbb{H}$}\label{sec:heinseberg_fourier}
Let us now recall spectral theory for the sub-Laplacian $\mathcal{L}$, harmonic analysis on $\Hei$, and an analog of the Fourier inversion theorem \eqref{e.convolution.R}. For $\lambda >0$, $\varepsilon = \pm 1$, and $j,k \in \mathbb{N}_{\geq 0}$, let us consider the following function 
\begin{equation}\label{eqn.e.function}
    e^{\lambda ,\varepsilon}_{k,j} (g):= \left\langle \pi_{ \varepsilon\frac{\lambda}{2k+1}} (g) \Phi_{k} , \Phi_{j} \right\rangle, \quad g\in \Hei,
\end{equation}
where $\pi_{\mu}$ is the Schr\"odinger representation corresponding to $\mu \in \R^{\ast}$, and $\Phi_{k}$ is the $k$-th Hermite function, that is,
\begin{equation*}\label{e.hermite.fun}
    \Phi_{k} (t):= \left( 2^{k} \sqrt{\pi}\, k! \right)^{-\frac{1}{2}} H_{k} (t) e^{-\frac{1}{2}t^{2}},
\end{equation*}
and $H_{k}$ denotes the $k$-th Hermite polynomial 
\begin{equation*}\label{e.hermite.pol}
    H_{k} (t):= (-1)^{k} e^{t^{2}} \frac{d^{k}}{dt^{k}} \left( e^{-t^{2}} \right).
\end{equation*}
The family $\{\Phi_{k}, \, k=0,1, \ldots \}$ forms an orthonormal basis for $L^{2} (\R, d\xi)$, where by $d\xi$ we denoted the Lebesgue measure, and the function $e^{\lambda ,\varepsilon}_{k,j}$ is the matrix element corresponding to the representation $\pi_{\varepsilon\frac{\lambda}{2k+1}}$. It is known that the functions $\{e^{\lambda ,\varepsilon}_{k,j}, \, \lambda>0, \, \varepsilon = \pm 1, \, j,k =0,1,\ldots \}$ are (generalized) eigenfunctions for the sub-Laplacian, that is, they satisfy
\begin{equation*}
    \mathcal{L} e^{\lambda ,\varepsilon}_{k,j} = (2k+1) \left| \varepsilon\frac{\lambda}{2k+1} \right| e^{\lambda ,\varepsilon}_{k,j} = \lambda \, e^{\lambda ,\varepsilon}_{k,j},
\end{equation*}
but they are not in $L^{2}(\Hei, dg)$, see \cite[p. 51]{Thangavelu98}. Let us introduce the following function, for $\lambda >0$, $\varepsilon = \pm1$, and $k\in \mathbb{N}_{\geq 0}$
\begin{align}
    & \varphi_{\lambda}^{k} (g) := \frac{\sqrt{\pi}}{2} \sum_{\varepsilon=\pm 1} e_{k,k}^{\lambda, \varepsilon}(g) \label{e.sum.eps} . 
\end{align}
which is also an eigenfunction with eigenvalue $\lambda$, where $L_{k}$ denotes the Laguerre polynomial of type 0, that is, 
\begin{align*}
    L_{k} (t) e^{-t} = \frac{1}{k!}\frac{d^{k}}{dt^{k}} \left( e^{-t} t^{k} \right) .
\end{align*}
Let us now collect some properties of the eigenfunctions $e_{k,j}^{\lambda, \varepsilon}$ and $\varphi^{k}_{\lambda}$.

\begin{prop}\label{prop.properties.e.functions}
    For $k,j \in \mathbb{N}_{\geq 0}$, $\lambda>0$, and $\varepsilon = \pm 1$, let $e_{k,j}^{\lambda, \varepsilon}$ be the eigenfunctions of $\mathcal{L}$ given by \eqref{eqn.e.function}. Then 
    \begin{align}
        & e_{k,k}^{\lambda, \varepsilon}(h^{-1}g) = \sum_{j=0}^{\infty} e_{k,j}^{\lambda, \varepsilon}(g) \overline{e_{k,j}^{\lambda, \varepsilon}(h)},\label{e.duplication}
        \\
        &  \varphi_{k}^{\lambda}(g)= L_{k} \left( \frac{1}{2} \frac{\lambda}{2k+1}  \vert z\vert^{2}\right) e^{-\frac{1}{4} \frac{\lambda}{2k+1} \vert z\vert^{2}} \cos \left( \frac{\lambda}{2k+1} t \right), \quad g= (z,t) \in \Hei,   \label{e.cov.alpha.k.H}
        \\
        & \sum_{j=0}^{\infty} \left|  e_{k,j}^{\lambda, \varepsilon}(g)  \right|^{2} = 1, 
    \end{align}
    for all $k\in \mathbb{N}_{\geq 0}, \, g,h\in \Hei,  \text{ and }  \varepsilon=\pm 1$.
\end{prop}

\begin{proof}
    For any $g,h \in \Hei$, we have that 
    \begin{align*}
        & e_{k,k}^{\lambda, \varepsilon}(h^{-1}g) = \left\langle \pi_{\varepsilon \frac{\lambda}{2k+1}} (h^{-1}g) \Phi_{k}  , \Phi_{k}  \right\rangle = \left\langle \pi_{\varepsilon \frac{\lambda}{2k+1}} (g) \Phi_{k}  , \pi_{\varepsilon \frac{\lambda}{2k+1}} (h) \Phi_{k}  \right\rangle 
        \\
        & = \left\langle \sum_{s=0}^{\infty} \left\langle \pi_{\varepsilon \frac{\lambda}{2k+1}} (g) \Phi_{k} , \Phi_{s} \right\rangle \Phi_{s} \, , \, \sum_{j=0}^{\infty} \left\langle \pi_{\varepsilon \frac{\lambda}{2k+1}} (h) \Phi_{k} , \Phi_{j} \right\rangle \Phi_{j}\right\rangle
        \\
        & = \sum_{s,j=0}^{\infty} \left\langle \pi_{\varepsilon \frac{\lambda}{2k+1}} (g) \Phi_{k} , \Phi_{s} \right\rangle \overline{\left\langle \pi_{\varepsilon \frac{\lambda}{2k+1}} (h) \Phi_{k} , \Phi_{j} \right\rangle} \langle\Phi_{s} , \Phi_{j} \rangle
        \\
        & =  \sum_{j=0}^{\infty} \left\langle \pi_{\varepsilon \frac{\lambda}{2k+1}} (g) \Phi_{k} , \Phi_{j} \right\rangle \overline{\left\langle \pi_{\varepsilon \frac{\lambda}{2k+1}} (h) \Phi_{k} , \Phi_{j} \right\rangle} =  \sum_{j=0}^{\infty} e_{k,j}^{\lambda, \varepsilon}(g) \overline{e_{k,j}^{\lambda, \varepsilon}(h)},
    \end{align*}
    where we used that Schr\"odinger representations are unitary, and that $\{\Phi_{k},  \, k=0,1,\ldots\}$ is an orthonormal basis for $L^{2} (\R, d\xi)$. By \cite[eq. 1.4.19 and 1.4.20]{Thangavelu98} and \cite[eq. 6.5 and 6.10 ]{Strichartz89} we have that, for $g= (z,t) \in \Hei$
    \begin{align*}
        &e^{\lambda, 1}_{k,k}(g) : = \left\langle \pi_{\frac{\lambda}{2k+1}} (g) \Phi_{k} , \Phi_{k} \right\rangle= \pi^{-\frac{1}{2}} e^{i\frac{\lambda}{2k+1} t} L_{k} \left( \frac{1}{2} \frac{\lambda}{2k+1}  \vert z\vert^{2}\right) e^{-\frac{1}{4} \frac{\lambda}{2k+1} \vert z\vert^{2}},
        \\ 
        & e^{\lambda, -1}_{k,k}(g) = \pi^{-\frac{1}{2}} e^{-i\frac{\lambda}{2k+1} t} L_{k} \left( \frac{1}{2} \frac{\lambda}{2k+1}  \vert z\vert^{2}\right) e^{-\frac{1}{4} \frac{\lambda}{2k+1} \vert z\vert^{2}},
    \end{align*}
    and \eqref{e.cov.alpha.k.H} follows. Lastly, 
    \begin{align*}
     \sum_{j=0}^{\infty} \left|  e_{k,j}^{\lambda, \varepsilon}(g)  \right|^{2} =  \sum_{j=0}^{\infty} \left| \left\langle \pi_{\varepsilon \frac{\lambda}{2k+1}} (g) \Phi_{k}  , \Phi_{j}  \right\rangle   \right|^{2} = \left| \left| \pi_{\varepsilon \frac{\lambda}{2k+1}}(g)\Phi_{k}\right| \right|^{2} = \Vert \Phi_{k}\Vert = 1,
     \end{align*}
    since $\pi_{\lambda}(g)$ is a unitary operator for all $g\in \Hei$ and $\lambda \neq 0$ and $\{ \Phi_{k}, \, k=0,1, \ldots\}$ is an orthonormal basis for  $L^{2}(\R , d\xi)$, and the proof is completed.
\end{proof}

\begin{rmk}
    One can view \eqref{e.duplication} as an analog on $\Hei$ of the addition formulae for spherical Harmonics $Y_{\ell m}$ on the sphere $\mathbb{S}^{2}$ and Bessel functions $J_{m}$ on $\R^{2}$, that is, 
    \begin{align*}
       & \sum_{m=-\ell}^{\ell} Y_{\ell m} (p) Y_{\ell m}(q)= \frac{2\ell+1}{4\pi} P_{\ell} \left(d_{\mathbb{S}^{2}} (p,q) \right), & p,q\in \mathbb{S}^{2},
        \\
       & \sum_{k=-\infty}^{\infty} J_{k} ( \vert x \vert ) J_{k} ( \vert y \vert ) \cos (\theta k) = J_{0} (\vert x-y\vert) , & x,y, \in \R^{2},
    \end{align*}
    where $P_{\ell}$ is the Legendre polynomial of order $\ell$, and $\theta$ is the angle between the vectors $x,y\in \R^{2}$.
\end{rmk}

On $\R^{n}$, a special role is played by isotropic fields, that is, fields whose covariance function is radial. Algebraically, these are fields whose covariance function is invariant under the action of the group of isometries fixing the origin, the rotation group. The group of isometries for $\Hei$ equipped with the Carnot-Carath\'eodory distance is given by the semidirect product $U(1)\rtimes \Hei$, where $U(1)\cong \mathbb{S}^{1}$ is the one-dimensional unitary group. The group $U(1)\rtimes \Hei$ acts on $\Hei$ by 
\begin{align*}
    & (e^{i\phi}, (z,t) )\cdot (w,s) := (z,t) (e^{i\phi}w, s),
\end{align*}
for any $(e^{i\phi},\; (z,t) )\in U(1)\rtimes \Hei, (w,s) \in \Hei$. In particular, any isometry can be written as a composition of the following maps
\begin{enumerate}
    \item left multiplication by $g\in \Hei$,  $h \rightarrow g^{-1}h$;
    
    \item rotations around the vertical component $(z,t) \rightarrow (e^{i\phi}z, t)$, for $\phi \in U(1)$;

    \item reflections $(z,t) \rightarrow\ (\bar{z}, -t)$,
\end{enumerate}
see \cite[Section 3.1]{Thangavelu98} and \cite{Strichartz89,Koranyi83, Isangulova19}. A function $f: \Hei \rightarrow \C$ is said to be radial if 
\begin{align*}
    f(z,t) = f(\vert z\vert ,t),
\end{align*}
for all $(z,t) \in\Hei$, that is, radial functions on $\Hei$ are functions invariant under the action of $U(1)$, which is the subgroup of $U(1)\rtimes \Hei$ fixing the identity. In analogy with radial functions in $\R^{n}$ being functions invariant under the action of the rotation group.

Let us introduce the function 
\begin{align}
    & F_{\lambda} (g) := \frac{8}{\pi^{2}} \sum_{k=0}^{\infty} \frac{1}{(2k+1)^{2}} \varphi^{k}_{\lambda} (g), \label{e.cov.H}
\end{align}
normalized so that $F_{\lambda} (e) =1$, where $e\in \Hei$ denotes the identity element. We have the following spectral decomposition on the Heisenberg group. 

\begin{thm}\label{thm.spectral.decomp}
    Let $f\in L^{2}(\Hei, dg)$, where $dg$ denotes the Haar measure. Then 
    \begin{align*}\label{e.spectral.decomp}
        f(g) = \int_{0}^{\infty} \left( \mathcal{P}_{\lambda }f \right)  (g) d\mu(\lambda),
    \end{align*}
    where $d\mu(\lambda)$ denotes the Plancherel measure on $\widehat{\Hei}= \R^{2} \, \dot{\cup} \, \R^{\ast}$
    \begin{align*}
        & d\mu(\lambda) := \frac{1}{(2\pi)^{2}} \vert \lambda \vert d\lambda \, 
    \end{align*}
    which is supported on $\R^{\ast}$, and 
    \begin{equation}\label{e.convolution.H}  \mathcal{P}_{\lambda} f = \frac{\pi^{2}}{4} f \ast F_{\lambda},
    \end{equation}
    where $F_{\lambda}$ is given by \eqref{e.cov.H}.
\end{thm}

\begin{proof}
    The proof follows from \cite[Theorem 2.1.1]{Thangavelu98} and we write it for completeness. By \cite[p. 68]{Thangavelu98} one has that 
    \begin{align*}
        \mathcal{P}_{\lambda} f (g) = \sum_{k=0}^{\infty}\frac{1}{(2k+1)^{2}} \sum_{\varepsilon = \pm 1}  f \ast \tilde{e}_{k}^{\varepsilon \lambda} (g),
    \end{align*}
    where the functions $\tilde{e}_{k}^{\varepsilon \lambda}$ are given by 
    \begin{align*}
        \tilde{e}_{k}^{\varepsilon \lambda}(g) = \tilde{e}_{k}^{\varepsilon \lambda}(z,t)  = e^{i \frac{\varepsilon \lambda}{2k+1} t }L_{k} \left( \frac{1}{2} \frac{\lambda}{2k+1}  \vert z\vert^{2}\right) e^{-\frac{1}{4} \frac{\lambda}{2k+1} \vert z\vert^{2}},
    \end{align*}
see \cite[p. 52 and eq. (1.4.20)]{Thangavelu98}. Thus,

    \begin{align*}
        &  \mathcal{P}_{\lambda} f (g) = \sum_{k=0}^{\infty}\frac{2}{(2k+1)^{2}}  f \ast \varphi^{k}_{ \lambda} (g) = \frac{\pi^{2}}{4} f\ast F_{\lambda}.
    \end{align*}
\end{proof}

\begin{rmk}\label{rmk.uniqueness}
The complex-valued Berry random field on $\R^{2}$ is the unique (in distribution) field satisfying the following properties 
\begin{enumerate}
    \item it is an isotropic standard complex Gaussian field;

    \item the trajectories are almost surely (generalized) eigenfunctions for the Laplace operator with eigenvalue $\lambda>0$;
\end{enumerate}
and it is characterized by its covariance function $2J_{0} (\sqrt{\lambda} | \cdot |)$, which is the unique smooth radial solution to the differential equation 
\begin{align*}
  &  \Delta_{\R^{n}} f + \lambda \, f =0 , & f(0)=2.
\end{align*}
Note that for the real-valued Berry random field the covariance is $J_{0} (\sqrt{\lambda} | \cdot |)$ and the initial condition $f(0)=1$. In light of the analogy between the Bessel function and the function $F_{\lambda}$ \eqref{e.convolution.R}, \eqref{e.convolution.H}, and \eqref{eq.complex.cov.berry}, it is reasonable to define a complex-valued Berry-Heisenberg random wave as a field having covariance $2F_{\lambda}$. One can then prove that such field would not be unique in the same way Berry random field is on Euclidean spaces. This is related to the fact that, on the Heisenberg group, there  are countably many independent radial solutions to the equation
\begin{align*}
    &  \mathcal{L} f + \lambda \, f =0 , & f(e)=2,
\end{align*}
as the functions $\{2\varphi^{k}_{\lambda}, \,k=0,1,\ldots\}$ and $f(g):= 2J_{0} ( \sqrt{\lambda} \,|z| )$ for $g=(z,t)$ all satisfy this equation, see \cite[Ch. 2]{Thangavelu98}. In the proof of Theorem \ref{thm.stat.field.H} it will be clear how the functions $\varphi^{k}_{\lambda} (\cdot)$ arise from non-equivalent unitary irreducible Schr\"odinger representations $\pi_{\lambda}$, while $J_{0} ( \sqrt{\lambda} \,\cdot )$ arises from the characters $\pi_{\beta}$. 
\end{rmk}

To take into account the existence of infinitely many independent deterministic radial eigenfunctions corresponding to the same eigenvalue, we introduce a function  $F_{\lambda}^{(c)}$ to extend $F_{\lambda}$. Let us consider the set
\begin{align}\label{def:O}
    & \mathcal{O}^{1}:= \left\{ c = (c_{k})_{k=-1}^{\infty} \, : \, c_{k} \geq 0 \text{ for all } k, \;   |c|:= \sum_{k=-1}^{\infty} c_{k}=1 \right\}
\end{align}
and for each $\lambda>0$, $c\in \mathcal{O}^{1}$, define 
\begin{align}
    & F_{\lambda}^{(c)} (g) :=c_{-1}  J_{0} ( \sqrt{\lambda} \,|z| ) + \sum_{k=0}^{\infty} c_{k} \, \varphi_{\lambda}^{k} (g). \label{e.cov.H.c}
\end{align}

In the next section we prove that for any fixed $\lambda>0$ and $c\in \mathcal{O}^{1}$, there exists a complex-valued random field $\{ \xi (g) ,\, g\in \Hei\}$ with covariance function $2F_{\lambda}^{(c)}$, or equivalently, a real-valued field with covariance $F_{\lambda}^{(c)}$. Moreover, the field is unique in  the following sense: if $Y$ is any centered complex-valued Gaussian unit-variance field whose law is invariant under isometries fixing the identity, such that $  \mathcal{L} Y + \lambda \, Y =0$, then $Y$ has covariance $F_{\lambda}^{(c)}$ for some $c\in \mathcal{O}^{1}$.

\section{Berry-Heisenberg random waves}\label{sec:berry-heisenbergRW}
The goal of this section is to define Berry-Heisenberg random waves. In Proposition \ref{prop.berry.horizontal} and  Proposition \ref{prop.approx.CLT} we connect them to the Euclidean Berry random wave model, and in Theorem \ref{thm.stat.field.H} we prove its existence and uniqueness.  

\subsection{Definition and connection to the classical Berry model}
We start with the definition of $(\alpha, c)$-Berry-Heisenberg random waves. 

\begin{definition}
     For $\alpha>0$ and $c\in \mathcal{O}^1$ fixed, we refer to the left-invariant Gaussian random field $\{  \xi_{\alpha,c} (g), \, g\in \Hei \}$ with zero mean and covariance 
     \begin{align*}
           \E\left[ \xi_{\alpha ,c} (g) \overline{\xi_{\alpha, c}(h)} \right] =2F_{\lambda}^{(c)} (h^{-1} g),
     \end{align*}
     as the  $(\alpha, c)$-Berry-Heisenberg random wave. 
\end{definition}

In Theorem \ref{thm.stat.field.H} we show that such field exists, is unique, and can be written as 
   \begin{equation*}\label{def:representation.r.f.H}
    \xi_{\alpha,c} (g) =  \sqrt{c_{-1}}\, b^{\C}_{\alpha} (z)+\sum_{k=0}^{\infty} \sqrt{c_{k}} \, \xi^{k}_{\alpha} (g) , \quad g=(z,t) \in \Hei,
\end{equation*}
where $ b^{\C}_{\alpha}$ is a complex Berry random wave and 
\begin{equation*}
     \xi^{k}_{\alpha} (g) := \pi^{\frac{1}{4}} \sum_{\varepsilon = \pm 1} \sum_{j=0}^{\infty} a^{\alpha ,\varepsilon}_{kj} e^{\alpha, \varepsilon}_{kj} (g) ,
\end{equation*}
where  $e^{\alpha, \varepsilon}_{kj}$ are the generalized $\mathcal{L}$-eigenfunctions given by \eqref{eqn.e.function} and $a^{\alpha, \varepsilon}_{kj}$ are complex-valued Gaussian random variables independent of $b^{\C}_{\alpha}$ and such that
\begin{align*}
    \E \left[ a^{\alpha, \varepsilon}_{kj} \overline{a^{\alpha, \varepsilon^{\prime}}_{k^{\prime}j^{\prime}}} \right] = \delta_{kk^{\prime}}\delta_{jj^{\prime}}\delta_{\varepsilon \varepsilon^{\prime}}.
\end{align*}
We refer to $ \xi^{k}_{\alpha}$ as the $k$-component of $\xi_{\alpha, c}$. 

\begin{rmk}
    Note that, for every $\alpha>0$, the two Gaussian random functions $g\rightarrow \xi_{\alpha ,c} (g)$ and $g\rightarrow \xi_{1 ,c} (\delta_{\sqrt{\alpha}} (g))$ have the same distribution, where $\delta$ denotes the dilation on $\Hei$. This is consistent with $b_{\alpha}^{\C} (z)$ and $b_{1}^{\C} (\sqrt{\alpha} \, z)$ having the same distribution. 
\end{rmk}

Let us now describe the connection between the fields $\{ \xi_{\alpha,c} (g) , \, g\in \Hei\}$ and $\{ \xi^{k}_{\alpha} (g) , \, g\in \Hei\}$ and the classical Berry random field in $\R^{2}$. First, we show that as $k\rightarrow \infty$, the family of fields  $\{ \xi^{k}_{\alpha} (g) , \, g\in \Hei\}_{k=0}^{\infty}$ converges to the Berry random field on $\R^{2}$ in the sense of finite-dimensional distributions.  Then, we prove that the field $\{ \xi_{\alpha,c} (g) , \, g\in \Hei\}$ can be described as the limit of superpositions of random waves, that is, $\mathcal{L}$-eigenfunction, in analogy with the Euclidean case  \cite{Berry1977}. Note that, by  Proposition \ref{prop.properties.e.functions}, the covariance function of the field $\{ \xi^{k}_{\alpha} (g), g\in \Hei\}$ is given by 
\begin{align*}
    \E\left[ \xi_{\alpha}^{k} (g) \overline{\xi_{\alpha}^{k}(h)} \right] = \sqrt{\pi} \sum_{\varepsilon = \pm 1} \sum_{j=0}^{\infty} e^{\alpha, \varepsilon}_{kj}(g) \overline{e^{\alpha, \varepsilon}_{kj}(h)} =2\varphi_{\alpha}^{k} (h^{-1 }g),
\end{align*}
and $ \xi_{\alpha}^{k} (g)$ is a complex Gaussian random variable with variance $2$.

\begin{prop}[Connection to classical Berry on the horizontal space]\label{prop.berry.horizontal}
    Let $\xi_{\alpha,c}$ be the $(\alpha, c)$-Berry-Heisenberg random fields, and $\xi_{\alpha}^{k}$ its $k$-th component. Then the sequence of random fields $\{\xi_{\alpha}^{k}(g), \, g\in\Hei ,  \, k \in \mathbb{N}_{\geq 0} \}$ converges to the complex-valued Berry random field $\{b^{\C}_{\alpha} ( \pi(g)), \, g\in\Hei   \}$, in the sense of finite-dimensional distributions, where $\pi(g)=z\in \R^{2}$ denotes the projection on the horizontal space.
\end{prop}

\begin{proof}
Note that $\xi_{\alpha}^{k}$ is a centered Gaussian field, so it is enough to prove that  
\begin{align*}
        \lim_{k\rightarrow \infty}  \E\left[ \xi_{\alpha}^{k} (g) \overline{\xi_{\alpha}^{k}(h)} \right]  =  \E\left[ b^{\C}_{\alpha} (z) \overline{b^{\C}_{\alpha}(w)} \right] ,
    \end{align*}
    for any $g=(z,t), h=(w,s) \in \Hei$. We have that   
    \begin{align*}
        & \E\left[ \xi_{\alpha}^{k} (g) \overline{\xi_{\alpha}^{k}(h)} \right] = 2\varphi_{\alpha}^{k} (h^{-1 }g) =2 L_{k} \left( \frac{1}{2} \frac{\alpha}{2k+1}  \vert z-w\vert^{2}\right) e^{-\frac{1}{4} \frac{\alpha}{2k+1} \vert z-w\vert^{2}} \cos \left( \frac{\alpha}{2k+1}( t-s) \right),
    \end{align*}
    and the result then follows from the following expansion for  Laguerre polynomials \cite[formula 8.1.8]{Szego39} 
    \[
    \lim_{n\rightarrow \infty} L_{n} (t/n)= J_{0} (2\sqrt{t})
    \]
    where $J_{0}$ denotes the Bessel function of first kind of order zero. Thus, 
    \begin{align*}
        & \lim_{k\rightarrow \infty}  \E\left[ \xi_{\alpha}^{k} (g) \overline{\xi_{\alpha}^{k}(h)} \right] =2 J_{0} \left(\sqrt{\alpha} \, | z-w |\right) =  \E\left[ b^{\C}_{\alpha} (z) \overline{b^{\C}_{\alpha}(w)} \right] .
    \end{align*}
\end{proof}

\begin{prop}\label{prop.approx.CLT}
    Let us fix $c\in \mathcal{O}^{1}$ and $N\in \mathbb{N}$ and let $\{ \phi_{\ell} , \, \theta^{\ell}_{k}, \, \psi^{\ell}_{j}, \; \ell = 1,\ldots N \}_{k,j=0}^{\infty}$ be a sequence of i.i.d. $\text{Unif}(0,2\pi)$, $\{ \varepsilon_{\ell}, \; \ell = 1,\ldots N\}$ be i.i.d. $\text{Unif}\{-1,1\}$, and $\{ \eta_{\ell}, \; \ell = 1,\ldots N\}$ be i.i.d. uniform random variables on $\mathbb{S}^{1}$. Moreover, we assume that $ \phi_{\ell} , \, \theta^{\ell}_{k}, \, \psi^{\ell}_{j}, \varepsilon_{\ell}, \eta_{\ell}$ are all independent from each other for all $k,j=0,1,\ldots,  \ell = 1,\ldots N$. Let us define the random wave 
    \begin{align*}
       & u_{\alpha}^{N} (g):= \frac{1}{\sqrt{N}} \sum_{\ell =1}^{N}  u_{\alpha, \ell}(g),
    \end{align*}
    where 
    \begin{align*}
        & u_{\alpha, \ell} (g) :=\sqrt{2c_{-1}}  e^{i \left( \sqrt{\alpha} \, z \cdot \eta_{\ell} + \phi_{\ell} \right)}+  \pi^{\frac{1}{4}} \sum_{k,j=0}^{\infty} \sqrt{2c_{k}} e^{i \left( \theta^{\ell}_{k} + \psi^{\ell}_{j}\right)} e^{\alpha, \varepsilon_{\ell}}_{k,j} (g), & g= (z,t) \in \Hei.
    \end{align*}
    Then, the finite dimensional distributions of $u_{\alpha}^{N}$ converge in law to the distribution of $\xi_{\alpha,c}$. 
\end{prop}

\begin{proof}
For each $g\in \Hei$, $\{u_{\alpha, \ell} (g)\}_{\ell=1}^{\infty}$ is a sequence of i.i.d. centered complex random variables, and hence the random vector 
    \begin{align*}
        \left(u_{\alpha}^{N} (g) , u_{\alpha}^{N} (h) \right) = \frac{1}{\sqrt{N}} \sum_{\ell =1}^{N} \left(u_{\alpha, \ell} (g) , u_{\alpha, \ell} (h) \right),
    \end{align*}
    converges in law to a centered complex Gaussian random variable with covariance matrix given by the covariance of the vector $\left(u_{\alpha, \ell} (g) , u_{\alpha, \ell} (h) \right)$. Note that, for $g=(z,t), h=(w,s) \in \Hei$
    \begin{align*}
        & \E\left[ u_{\alpha, \ell}(g) \overline{ u_{\alpha, \ell}(h)} \right] 
        \\
        &= 2c_{-1} \E\left[ e^{i \sqrt{\alpha} \,( z -w) \cdot \eta_{\ell}  } \right] + 2 \sqrt{\pi} \sum_{k, k^{\prime}, j,j^{\prime}=0}^{\infty}\sqrt{c_{k}} \sqrt{c_{k^{\prime}}}\E\left[e^{i \left( \theta^{\ell}_{k} + \psi^{\ell}_{j}\right)} e^{\alpha, \varepsilon_{\ell}}_{k,j} (g)  e^{-i \left( \theta^{\ell}_{k^{\prime}} + \psi^{\ell}_{j^{\prime}}\right)} \overline{e^{\alpha, \varepsilon_{\ell}}_{k^{\prime},j^{\prime}} (h) } \right]
        \\
        & =2c_{-1} \E\left[ e^{i \sqrt{\alpha} \,( z -w) \cdot \eta_{\ell}  } \right] +2 \sqrt{\pi}  \sum_{k, k^{\prime}, j,j^{\prime}=0}^{\infty}\sqrt{c_{k}} \sqrt{c_{k^{\prime}}}\E\left[e^{i \left( \theta^{\ell}_{k}  - \theta^{\ell}_{k^{\prime}}\right) } \right]\E\left[e^{i \left( \psi^{\ell}_{j}  - \psi^{\ell}_{j^{\prime}}\right) } \right] \E\left[e^{\alpha, \varepsilon_{\ell}}_{k,j} (g)   \overline{e^{\alpha, \varepsilon_{\ell}}_{k^{\prime},j^{\prime}} (h) } \right]
        \\
        & =2c_{-1} \E\left[ e^{i \sqrt{\alpha} \,( z -w) \cdot \eta_{\ell}  } \right] +\sqrt{\pi}    \sum_{k, j=0}^{\infty}c_{k} \left( e^{\alpha, 1}_{k,j} (g)   \overline{e^{\alpha, 1}_{k,j} (h) } +  e^{\alpha, -1}_{k,j} (g)   \overline{e^{\alpha, -1}_{k,j} (h) }    \right) 
        \\
        &  =2c_{-1} \E\left[ e^{i \sqrt{\alpha} \,( z -w) \cdot \eta_{\ell}  } \right] + \sqrt{\pi} \sum_{k=0}^{\infty} c_{k} \sum_{\varepsilon = \pm 1}  \sum_{j=0}^{\infty}  e^{\alpha, \varepsilon}_{k,j} (g)   \overline{e^{\alpha, \varepsilon}_{k^{\prime},j^{\prime}} (h) } 
        \\
        & =  2c_{-1} \E\left[ e^{i \sqrt{\alpha} \,( z -w) \cdot \eta_{\ell}  } \right] + \sqrt{\pi} \sum_{k=0}^{\infty} c_{k} \sum_{\varepsilon = \pm 1}    e^{\alpha, \varepsilon}_{k,k} (h^{-1}g)   
        \\
        & =  2c_{-1} J_{0} \left( \sqrt{\alpha} \, \vert z -w \vert  \right) +2 \sum_{k=0}^{\infty} c_{k}   \varphi^{k}_{\alpha} (h^{-1}g) = 2F_{\alpha}^{(c)}  (h^{-1}g),
    \end{align*}
    where we used Proposition \ref{prop.properties.e.functions} and the following representation of the Bessel function 
    \begin{align*}
        & J_{0} \left( \sqrt{\alpha} \, \vert z \vert\right) =  \E\left[ e^{i \sqrt{\alpha} \,z \cdot X  } \right],
    \end{align*}
    for any $z\in \R^{2}$, where $X$ is uniformly distributed on $\mathbb{S}^{1}$.
\end{proof}

\subsection{Existence and uniqueness}

The aim of this section is to prove existence and uniqueness of the $(\alpha,c)$-Berry-Heisenberg random field. First, let us prove that if a Gaussian field has covariance function given by $F^{(c)}_{\alpha}$ for some  $c\in \mathcal{O}^{1}$ and $\alpha>0$, then its trajectories are almost surely smooth. The fact that $F^{(c)}_{\alpha}$ is indeed a covariance function will follow from Theorem \ref{thm.stat.field.H}. Existence and uniqueness of the field are the content of Theorem \ref{thm.stat.field.H}, where we also prove that  the sample trajectories are generalized eigenfunctions for the sub-Laplacian $\mathcal{L}$.

\begin{thm}\label{thm.smooth.field}
    Let $\alpha>0$ and $c\in \mathcal{O}^{1}$ be fixed, and assume that $\xi : \Hei \rightarrow \mathbb{C}$ is a stationary Gaussian random field with covariance function $F^{(c)}_{\alpha}$ given by \eqref{e.cov.H.c}. Then there exists a modification of $\{ \xi (g) , \, g\in \Hei\}$ such that, for every $n,m \in \mathbb{N}$, the functions 
    \begin{align*}
        & g \longrightarrow X^{n}_{g} \xi (g) , & g \longrightarrow Y^{n}_{g}\xi(g), 
        \\
        & g \longrightarrow Z^{n}_{g}\xi(g), & g \longrightarrow X^{n}_{g}Y^{m}_{g} \xi (g),
    \end{align*}
    are almost surely continuous, and the random variables
    \begin{align*}
        & X^{n}_{g} \xi (g), &Y^{n}_{g}\xi(g),  && Z^{n}_{g}\xi(g),&&X^{n}_{g}Y^{m}_{g} \xi (g)
    \end{align*}
    are in  $L^{2} (\Omega, \Prob)$ for all $g\in \Hei$, where $X^{n}_{g} := X \circ \dots \circ X$ $n$-times.
\end{thm}

\begin{proof}
    Let us fix $\varepsilon>0$. We first prove that there exists a modification of $\{ \xi (g) , \, g\in \Hei\}$ such that all its paths are H\"older continuous of all orders $0  < \beta < \frac{\varepsilon}{4+\varepsilon}$. By Kolmogorov-Chenstov continuity Theorem on metric measure spaces \cite[Theorem 1.1]{KratschmerUrusov23} applied to the Heisenberg group with the Carnot-Carath\'eodory distance, it is enough to prove that, for all $g,h\in \Hei$
\begin{align*}
    \E \left[ \vert \xi (g) - \xi(h)\vert^{4+\varepsilon} \right] \leq A \, d_{cc}(g,h)^{4+\varepsilon},
\end{align*}
for some constant $A$ independent of $g$ and $h$. We claim that there exists a positive constant $A$ such that, for all $g,h\in \Hei$
\begin{align}\label{eqn.claim.to.do}
    &  \E \left[ \vert \xi (g) - \xi(h)\vert^{2} \right] \leq A \; \alpha \; d_{cc} (g,h)^{2}.
\end{align}
Let us assume \eqref{eqn.claim.to.do} for now. The random variable $\xi(g)-\xi(h)$ is Gaussian with variance $\sigma_{g,h}^{2}:= \E \left[ \vert \xi (g) - \xi(h)\vert^{2} \right]$, so that 
\begin{align*}
    & \E \left[ \vert \xi (g) - \xi(h)\vert^{4+\varepsilon} \right] = \sigma_{g,h}^{4+\varepsilon} \,\Gamma\left( 1+ \frac{4+\varepsilon}{2} \right) = \sigma_{g,h}^{4+\varepsilon} \,\Gamma\left( 3+\frac{\varepsilon}{2} \right)
    \\
    & 
    = (\sigma_{g,h}^2)^{\frac{4+\varepsilon}{2}} \Gamma\left( 3+\frac{\varepsilon}{2} \right)=\E \left[ \vert \xi (g) - \xi(h)\vert^{2} \right]^{\frac{4+\varepsilon}{2}} \Gamma\left(3+\frac{\varepsilon}{2}\right) 
    \\&\leq  \Gamma\left(3+\frac{\varepsilon}{2}\right) \left( A \; \alpha \right)^{2+\frac{\varepsilon}{2} } d_{cc} (g,h)^{4+ \varepsilon},
\end{align*}
which proves that there exists a modification of $\{ \xi (g) , \, g\in \Hei\}$ with H\"older trajectories. Let us now prove \eqref{eqn.claim.to.do}. We have that 
\begin{align*}
    &\E[|\xi(g)-\xi(h)|^2]=2(F^{(c)}_\alpha(e)-F^{(c)}_\alpha(h^{-1}g))
    \\
    &=2c_{-1} \left( 1-J_{0} (\sqrt{\lambda} |z| ) \right) + \sum_{k=0}^{\infty}2c_{k} (\varphi^k_{\alpha}(e)-\varphi^k_{\alpha}(h^{-1}g)) ,
\end{align*}
where 
\begin{align*}
    & \vert \varphi^k_{\alpha}(e)-\varphi^k_{\alpha}(h^{-1}g) \vert =  \left| 1- \cos \left(\frac{\alpha}{2k+1} t \right)e^{-\frac{1}{4}\frac{\alpha}{2k+1}|z|^2} L_{k} \left( \frac{\alpha}{2(2k+1)}|z|^2) \right)  \right|
    \\
    & \leq \left| 1-  e^{-\frac{1}{4}\frac{\alpha}{2k+1}|z|^2} L_{k} \left( \frac{\alpha}{2(2k+1)}|z|^2) \right)  \right| 
    \\
    & + \left| \left(1- \cos \left(\frac{\alpha}{2k+1} t \right) \right) e^{-\frac{1}{4}\frac{\alpha}{2k+1}|z|^2} L_{k} \left( \frac{\alpha}{2(2k+1)}|z|^2) \right)  \right|
    \\
    & =: I+II,
\end{align*}
for $h^{-1}g=(z,t) \in \Hei$. Set $f_{k}(s) := e^{-\frac{s}{2}} L_{k}(s)$, for $s>0$. Then 
\begin{align*}
    & f_{k}^{\prime} (s) = - e^{-\frac{s}{2}} \left( \frac{1}{2} L_{k}(s) + L_{k-1}^{(1)} (s) \right),
    \\
    & \vert 1- f_{k}(s) \vert = \vert f_{k}(0) - f_{k}(s) \vert \leq |s| \sup_{0\leqslant u \leqslant s} \vert f_{k}^{\prime} (u) \vert  \leq  |s| \frac{2k+1}{2},
\end{align*}
where we used the bound 
\begin{align*}
    \vert e^{-\frac{s}{2}} L_{k}^{(n)}(s) \vert \leq  L_{k}^{(n)}(0) = \binom{k+n}{k},
\end{align*}
which holds uniformly in $n,k\in \mathbb{N}$, $s\in \R$, and hence 
\begin{align*}
    I \leq \frac{\alpha}{2(2k+1)}|z|^2 \frac{2k+1}{2} = \frac{\alpha}{4} |z|^2.
\end{align*}
On the other hand, $II$ can be easily bounded by
\begin{align*}
    &  II \leq \left| 1- \cos \left(\frac{\alpha}{2k+1} t \right)  \right| \leq \frac{\alpha}{2k+1} |t|\leq \alpha \, |t|,
\end{align*}
so that 
\begin{align*}
    & \vert \varphi^k_{\alpha}(e)-\varphi^k_{\alpha}(h^{-1}g) \vert \leq  \alpha \left( |z|^2+|t| \right) \leq c \, \alpha \, d_{cc}(g,h)^{2} .
\end{align*}
Moreover, 
\begin{align*}
    & \left| 1-J_0(\sqrt{\lambda}|z|)\right| \leq \frac{\lambda|z|^2}{2}\leq \lambda \, d_{cc}(g,h)^2
\end{align*}
for all $h^{-1}g=(z,t) \in \Hei$ when $|z|$ is small. Combining everything together, then  \eqref{eqn.claim.to.do} follows since $(c_k)_{k=-1}^{\infty} \in \mathcal{O}^{1}$. Note that \eqref{eqn.claim.to.do} can also be proven using  Taylor formula on Carnot groups.

We now prove that the derivatives of the field are continuous.  Note that the function $F^{(c)}_{\alpha}$ satisfies $\left( \mathcal{L} + \alpha \right) F^{(c)}_{\alpha} =0$ and is then in $C^\infty(\mathbb{H})$ since $\mathcal{L}$ is hypoelliptic thanks to H\"ormander's condition. In particular, for each $g\in\Hei$
\begin{align*}
    \E \left[ X_{g} \xi (g) \overline{X_{g} \xi (g)} \right] = X_{h} X_{g} F^{(c)}_{\alpha}(h^{-1} g)_{|_{g=h}} < \infty ,
\end{align*}
and hence the field $\{ \xi (g) , \, g\in \Hei\}$ is  mean-square differentiable and its derivative $\{ X_{g}\xi (g), \, g\in \Hei\}$ is yet another stationary Gaussian field with finite variance \cite[Ch. 1 par. 5]{Yaglom87}. Almost sure continuity of the function $g\rightarrow X_{g}\xi (g)$ then follows again by \cite[Theorem 1.1]{KratschmerUrusov23} applied to the derivative field $X_{g}\xi (g)$. This follows as in the case of $\xi( g)$ after we prove that 
\begin{align}\label{eqn.der.field.cov}
    &  \E \left[ \vert  X_{g}\xi (g) -  X_{h}\xi (h)\vert^{2} \right] \leq c_{\alpha} \; d_{cc} (g,h)^{2},
\end{align}
for some constant $c_{\alpha}>0$. Note that $X_{g}\xi (g)$ is a stationary Gaussian field with covariance function 
\begin{align*}
    & K_\eta(g,h) =  \E \left[ X_{g} \xi (g) \overline{X_{h} \xi (h)} \right] = X_{g} X_{h} F^{(c)}_{\alpha}(h^{-1} g),
\end{align*} 
which is a smooth function. Moreover, if $K(g):=K(g,e)$ is the covariance function of a stationary field on $\Hei$ then $\nabla_{H}K (e)=0$, and hence by the Taylor formula on Carnot groups, see e.g. \cite[Example 20.2.5]{BonfiglioliLanconelliUguzzoni2007} it follows that 
\begin{align*}
    & K(g)=K(e) + \langle \nabla_{\mathbb{H}} K (e), z \rangle + O(d_{cc}(g,e)^2)= K(e) + O(d_{cc}(g,e)^2) , & g=(z,t)\in\Hei,
\end{align*}
that is,
\begin{align*}
    \vert K(e) - K(g) \vert \leq c \, d_{cc}(g,e)^{2},
\end{align*}
for some constant $c>0$, and \eqref{eqn.der.field.cov} is proven. The argument for other derivatives is similar. 
\end{proof}

\begin{thm}\label{thm.stat.field.H}
    Let $\alpha>0$ be fixed. Then the following are equivalent.

    \begin{enumerate}
        \item  There exists a left-invariant random field $\xi : \Hei \rightarrow \mathbb{C}$ such that 

        \begin{enumerate}[label=(\roman*)]
        \item $\xi$ is a Gaussian field and the random variable $\xi(g)$ is $N_{\mathbb{C}}(0,2)$ for all $g\in\Hei$;
        \item the function $g\rightarrow \xi (g)$ is smooth a.s. and it is a generalized eigenfunction for the sub-Laplacian with eigenvalue $\alpha$;
        \item the law of $\xi$ is invariant under isometries that fix the identity.
    \end{enumerate}

        \item $\xi_{\alpha}$ is a complex-valued Gaussian field whose trajectories are smooth a.s. with mean zero and covariance function $2F_{\alpha}^{(c)}$ for some $c\in  \mathcal{O}^{1}$.

        \item The field $\xi$ can be represented as
\begin{equation}\label{representation.r.f.H}
    \xi (g) =  \sqrt{c_{-1}}\, b^{\C}_{\alpha} (z)+\sum_{k=0}^{\infty} \sqrt{c_{k}} \, \xi^{k}_{\alpha} (g) , \quad g=(z,t) \in \Hei,
\end{equation}
for some $c\in  \mathcal{O}^{1}$, where $\{ b^{\C}_{\alpha} (z) , z \in \R^{2} \}$ is a complex-valued Berry random field on $\R^{2}$ and 
\begin{align}\label{e.field.component}
    & \xi^{k}_{\alpha} (g) := \pi^{\frac{1}{4}} \sum_{\varepsilon = \pm 1} \sum_{j=0}^{\infty} a^{\alpha ,\varepsilon}_{kj} e^{\alpha, \varepsilon}_{kj} (g) ,
\end{align}
where $e^{\alpha, \varepsilon}_{kj}$ are the generalized $\mathcal{L}$-eigenfunctions given by \eqref{eqn.e.function} and $a^{\alpha, \varepsilon}_{kj}$ are complex-valued Gaussian random variables independent of $b^{\C}_{\alpha}$ and such that
\begin{align*}
    \E \left[ a^{\alpha, \varepsilon}_{kj} \overline{a^{\alpha, \varepsilon^{\prime}}_{k^{\prime}j^{\prime}}} \right] = \delta_{kk^{\prime}}\delta_{jj^{\prime}}\delta_{\varepsilon \varepsilon^{\prime}}.
\end{align*}
    \end{enumerate}
    \end{thm}

\begin{proof}
    $(1) \implies (2)$ Note that by left invariance of the field the covariance function satisfies 
    \begin{align*}
        & C(g,h) := \E \left[ \xi_{\alpha} (g) \overline{\xi_{\alpha}(h)} \right]=  \E \left[ \xi_{\alpha} (h^{-1}g) \overline{\xi_{\alpha}(e)} \right] = C \left( h^{-1}g , e \right),
    \end{align*}
    so it is enough to consider the function 
    \begin{align*}
        & C(g) := \E \left[ \xi_{\alpha} (g) \overline{\xi_{\alpha}(e)} \right].
    \end{align*}
    By assumptions on $\xi_{\alpha}$ it then follows that $C$ is a smooth function on $\Hei$ satisfying 
    \begin{align*}
        & - \mathcal{L} C(g) = \alpha C(g),
    \end{align*}
    and because the field is invariant by the action of $U(1)$, we have that $C(g) = C(r,t)$ for $g=(z,t)\in \Hei$ and $r:= |z|$.  Note that the sub-Laplacian in spherical coordinates $(r,\theta, t)$ is given by 
    \begin{align*}
        & - \mathcal{L} = - \mathcal{L}^{rad} - \frac{1}{r^{2}} \partial^{2}_{\theta} + \partial^{2}_{\theta t}, & \mathcal{L}^{rad} = \partial^{2}_{r} + \frac{1}{r} \partial_{r} - \frac{1}{4} r^{2} \partial^{2}_{t},
    \end{align*}
    so that 
    \begin{align*}
        & - \mathcal{L}^{rad} C (r,t) = \alpha C(r,t).
    \end{align*}
    Let us consider the Fourier transform in the $t$ component
    \begin{align*}
        \hat{C} (r, \mu) := \int_{\R} C(r,t) e^{- i\mu t} dt, 
    \end{align*}
    and rewrite the radial sub-Laplacian as 
    \begin{align*}
        & \mathcal{L}^{rad} = \partial^{2}_{r} + \frac{1}{r} \partial_{r} - \frac{1}{4} r^{2} \mu^{2}.
    \end{align*}
    For $\mu \in \R$ and $r\geq 0$, the function  $g_{\mu} (r) :=  \hat{C} (r, \mu)$  satisfies the ordinary differential equation
    \begin{align}\label{eqn.dif.eq.1}
        & g^{\prime \prime}_{\mu}(r) + \frac{1}{r} g^{\prime}_{\mu} (r) + \left( \alpha - \frac{1}{4} r^{2} \mu^{2}\right) g_{\mu} (r) =0.
    \end{align}
    Note that if  $\mu =0$, we recover the Bessel equation 
    \begin{align*}
        & r^{2} g^{\prime \prime}_{0}(r) + r g^{\prime}_{0} (r) +  \alpha r^{2}  g_{0} (r) =0,
    \end{align*}
    whose solution satisfying $g_{0}(0)=2$ is given by the Bessel function of first kind $g_{0}(r)= 2  J_{0}(\sqrt{\alpha} r)$. The boundary condition follows from $\xi_{\alpha} (g) \sim N_{\C} (0,2)$. For $\mu \neq 0$ let us consider the following change of variable 
    \begin{align*}
       & v(r) := e^{\frac{r}{2}} g_{\mu} \left( \sqrt{\frac{2}{|\mu|} r} \right), & g_{\mu}(r)= e^{- \frac{1}{4} |\mu| r^{2}} v\left( \frac{1}{2}|\mu| r^{2} \right),
    \end{align*}
    which then solves 
    \begin{align}\label{eqn.eq.diff.Lag}
        & \frac{1}{2} |\mu| r^{2} v^{\prime \prime} \left( \frac{1}{2}|\mu| r^{2} \right)+ \left(1 -\frac{1}{2}|\mu| r^{2}  \right)v^{\prime } \left( \frac{1}{2}|\mu| r^{2} \right) + \frac{\alpha-|\mu|}{2|\mu|}  v \left( \frac{1}{2}|\mu| r^{2} \right) =0.
    \end{align}
    Note that \eqref{eqn.eq.diff.Lag} is a Laguerre type differential equation and it admits bounded solutions if and only if 
    \begin{align*}
        & k:= \frac{\alpha-|\mu|}{2| \mu|}   \in \mathbb{N}_{0},
    \end{align*}
    that is, if and only if $| \mu| = \frac{\alpha}{2k+1}$, and in that case the solution is the Laguerre polynomial of order $k$. We then have that  
    \begin{align*}
       & g_{\mu}(r)= 2 J_{0}(\sqrt{\alpha} r), & \text{if } \mu=0,
       \\
       &g_{\mu} (r) = e^{- \frac{1}{4}\frac{\alpha}{2k+1} r^{2}} L_{k}\left( \frac{1}{2}\frac{\alpha}{2k+1} r^{2} \right), & \text{if } \vert \mu \vert =\frac{\alpha}{2k+1} ,
       \\
       & g_{\mu} (r) =0, & \text{otherwise,}
    \end{align*}
    and hence 
    \begin{align*}
        \hat{C}(r,\mu) =2c_{-1}  J_{0}(\sqrt{\alpha} r)\delta\left(  \mu   \right)+  \sum_{k=0}^{\infty} c_{k} e^{- \frac{1}{4}\frac{\alpha}{2k+1} r^{2}} L_{k}\left( \frac{1}{2}\frac{\alpha}{2k+1} r^{2} \right)\delta\left( \vert \mu \vert - \frac{\alpha}{2k+1}  \right),
    \end{align*}
    for some $c_{k} \in \R$. By using Fourier inversion formula in the $t$ direction we have that 
    \begin{align*}
        &C(g)=C(r,t)=\int_{\R} \hat{C}(r,\mu) e^{i\mu t} d\mu 
        \\
        &=2c_{-1}  J_{0}(\sqrt{\alpha} r)+ \sum_{k=0}^{\infty} c_{k} e^{- \frac{1}{4}\frac{\alpha}{2k+1} r^{2}} L_{k}\left( \frac{1}{2}\frac{\alpha}{2k+1} r^{2} \right) 2 \cos \left(\frac{\alpha}{2k+1} t \right) =2F^{(c)}_{\alpha} (g).
    \end{align*}
    Note that $\sum_{k=-1}^{\infty} c_{k} =1$ because $\xi_{\alpha} (g) \sim N_{\C} (0,2)$. Let us now prove that $c_{k} \geq 0$ for all $k\in \mathbb{N}$. Let us recall that $C: \Hei \rightarrow \R$ is non-negative definite if and only if 
    \[
    \sum_{i,j=1}^{n} a_{i} \overline{a}_{j} C( g_{j}^{-1} g_{i} ) \geq 0
    \]
    for any $n\in \mathbb{N}$, $a_{1}, \ldots , a_{n} \in \mathbb{C}$, and any $g_{a}, \ldots , g_{n} \in \Hei$. By choosing points of the form $g_{i} = (0, t_{i})$ it then follows that the function 
    \begin{align*}
      &  f_{\alpha}: \R \longrightarrow \R , &  f_{\alpha} (t) := 2 \sum_{k=0}^{\infty} c_{k} \cos \left(\frac{\alpha}{2k+1} t \right),
    \end{align*}
    is non-negative definite. By Bochner theorem, its Fourier transform 
    \begin{align*}
      & \hat{f}_{\alpha} (\mu) = 2 \pi \sum_{k=0}^{\infty} c_{k} \delta\left( \vert \mu \vert - \frac{\alpha}{2k+1}  \right)
    \end{align*}
    is then a measure on $\R$. If there exists a $k_{0}$ such that $c_{k_{0}} <0$, then it would follow that 
    \begin{align*}
        &  \hat{f}_{\alpha} \left( \left\{ \frac{\alpha}{2k_{0}+1} \right\} \right) = 2\pi c_{k_{0}} <0 ,
    \end{align*}
    which contradicts $\hat{f}_{\alpha}$ being a measure. One can prove that $c_{-1}>0$ in a similar way, using that $J_{0}$ is a covariance function.

     $(2) \implies (3)$  By the Stone-von Neumann Theorem \ref{thm.stone.neuman} we know that the unitary dual is given by $\widehat{\Hei}=\R^{2} \,  \dot{\cup} \, \R^{\ast}$, where for $\beta \in \R^{2}$ $\pi_{\beta}$'s are the characters acting on $\C$, and  for $\lambda \neq 0$, $\pi_{\lambda}$'s are the  Schr\"odinger representations acting on  the same representation space $H_{\pi} = L^{2} (\R,d\xi)$, where $d\xi$ denotes the Lebesgue measure. By Theorem \ref{thm.stationarygroup} and \eqref{eqn.cov.components},  we can write the covariance function of $\xi_{\alpha}$ as
    \begin{align}\label{Bg}
        & B(g) =\int_{\R^{2}} \pi_{\beta} (g)F_{-1}(d\beta)+ \int_{\R} \sum_{k,j=0}^{\infty} \left\langle  \pi_{\lambda} (g) \Phi_{k}, \Phi_{j} \right\rangle  F_{jk}(d\lambda),
    \end{align}
    where $\{\Phi_{k}, \, k=0,1,\ldots\}$ is any orthonormal basis for $L^{2} (\R,d\xi)$, and $\{F_{jk}, \, j,k=0,1,\ldots\}$ are measures on $\R$ satisfying 
    \begin{align*}
        & \sum_{k=0}^{\infty} F_{kk} (\R) < \infty , & F_{kj} = F_{jk},
    \end{align*}
    and $F_{-1}$ is a finite measure on $\R^{2}$. For a fixed $\alpha>0$ and $c\in \mathcal{O}^{1}$, let us consider the following family of measures 
    \begin{align}\label{e.choice.measure}
        & F_{jk} (d\lambda):=\sqrt{\pi}c_{k} \delta_{jk} \sum_{\varepsilon = \pm 1}  \delta_{\varepsilon\frac{\alpha}{2k+1}} \left( \lambda  \right), 
        \\
        & F_{-1} (d\beta) := 2c_{-1} d\sigma^{\alpha} (\beta),
    \end{align}
    where $\delta_{\varepsilon\frac{\alpha}{2k+1}} $ denotes the Dirac delta function centered at $\varepsilon\frac{\alpha}{2k+1}$, and $\sigma^{\alpha}$ is the uniform measure on the circle of radius $\frac{\sqrt{\alpha}}{2\pi}$. Note that 
    \begin{align*}
      &  F_{kk} ( \R) =2\sqrt{\pi} c_{k} ,& F_{-1} (\R^{2}) = 2c_{-1},
    \end{align*}
and hence condition \eqref{e.condition.pos.def} is satisfied since $c\in \mathcal{O}^{1}$. We now show that the covariance function corresponding to this choice of $F_{jk}$ and $F_{-1}$ coincides with $F^{(c)}_{\alpha}$, and that the corresponding field $\xi_{\alpha}$ admits the representation \eqref{representation.r.f.H}. The covariance function \eqref{Bg} can  be written as 
\begin{align*}
    & B(g) =\int_{\R^{2}} \pi_{\beta} (g)F_{-1}(d\beta) +  \int_{\R} \sum_{k,j=0}^{\infty} \langle \pi_{\lambda} (g) \Phi_{k}, \Phi_{j} \rangle dF_{jk} (\lambda) = 
    \\
    & =2c_{-1}\int_{\R^{2}}e^{2\pi i \, \beta \cdot z}d\sigma^{\alpha} (\beta) + \sqrt{\pi}\int_{\R} \sum_{k,j=0}^{\infty} \langle \pi_{\lambda} (g) \Phi_{k}, \Phi_{j} \rangle c_{k} \delta_{jk} \sum_{\varepsilon = \pm 1}  \delta_{\varepsilon\frac{\alpha}{2k+1}} \left( \lambda  \right)
    \\
    & = 2c_{-1}J_{0} ( \sqrt{\alpha } \, \vert z \vert) + \sqrt{\pi}\sum_{k=0}^{\infty}c_{k} \sum_{\varepsilon = \pm 1} \int_{\R}  \langle \pi_{\lambda} (g) \Phi_{k}, \Phi_{k} \rangle \delta_{\varepsilon\frac{\alpha}{2k+1}} \left( \lambda  \right) 
    \\
    &= 2 c_{-1}J_{0} ( \sqrt{\alpha } \, \vert z \vert) +\sqrt{\pi} \sum_{k=0}^{\infty}c_{k} \sum_{\varepsilon = \pm 1}  \langle \pi_{\varepsilon \frac{\alpha}{2k+1}} (g) \Phi_{k}, \Phi_{k} \rangle
    \\
    & =2 c_{-1}J_{0} ( \sqrt{\alpha } \, \vert z \vert)+  \sum_{k=0}^{\infty}c_{k} \sqrt{\pi}\sum_{\varepsilon = \pm 1} e^{\alpha, \varepsilon}_{kk}(g) 
    \\
    &= 2c_{-1}J_{0} ( \sqrt{\alpha } \, \vert z \vert)+ 2 \sum_{k=0}^{\infty} c_{k} \varphi_{\alpha}^{k} (g) =:2 F_{\alpha}^{(c)} (g),
\end{align*}
where we used the following representation of the Bessel function 
\begin{align*}
    &\int_{\R^{2}}e^{2\pi i \, \beta \cdot z}d\sigma^{\alpha} (\beta) = J_{0} ( \sqrt{\alpha } \, \vert z \vert).
\end{align*}
Let us now represent the corresponding field $\xi_{\alpha}$.   By \eqref{e.stationary.r.f.group} we have that 
\begin{align*}
   & \xi_{\alpha} (g) = \int_{\R^{2}} \pi_{\beta} (g)dZ_{-1}(\beta) + \int_{\R} \sum_{k,j=0}^{\infty}\left\langle \pi_{\lambda} (g) \Phi_{k}, \Phi_{j} \right\rangle dZ_{jk} (\lambda)
\end{align*}
    where \{$Z_{jk}, j,=0,1,\ldots\}$ is a family of random measures on $\R^{\ast}$ satisfying 
\begin{align*}
    \E \left[Z_{ki} (A) \overline{Z_{jl} (B)} \right] = \delta_{il} F_{kj}(A\cap B) = \sqrt{\pi} \delta_{il} \delta_{jk}  c_{j} \sum_{\varepsilon = \pm 1}  \delta_{\varepsilon \frac{\alpha}{2k+1}} (A\cap B),
\end{align*}
for any measurable sets $A, B \subset \R^{\ast}$, and $Z_{-1}$ is a random measure on $\R^{2}$ satisfying 
\begin{align*}
    \E \left[Z_{-1} (A) \overline{Z_{-1} (B)} \right] =  F_{-1}(A\cap B) = 2c_{-1} \sigma^{\alpha}  (A\cap B),
\end{align*}
for any measurable sets $A, B \subset \R^{2}$, where $\sigma^{\alpha}$ is the uniform measure on the circle of radius $\frac{\sqrt{\alpha}}{2\pi}$. Note that 
\begin{align*}
    & \int_{\R^{2}} \pi_{\beta} (g)dZ_{-1}(\beta)  = \int_{\R^{2}} e^{2\pi i \, \beta \cdot z} dZ_{-1}(\beta)
\end{align*}
is a Gaussian field on $\R^{2}$ with covariance function given by $2c_{-1}J_{0} ( \sqrt{\alpha } \, \vert \cdot \vert)$ and hence 
\begin{align*}
    & \int_{\R^{2}} \pi_{\beta} (g)dZ_{-1}(\beta)  \stackrel{(d)}{=} \sqrt{c_{-1}} \, b^{\C}_{\alpha} (z), & g=(z,t) \in \Hei.
\end{align*}
Note that each $Z_{jk}$ is a Gaussian random measure since $\xi_{\alpha}$ is a Gaussian field. In particular, if $A$ does not contain $\frac{\alpha}{2k+1}$ nor $-\frac{\alpha}{2k+1}$, then 
\begin{align*}
    &   \E \left[ \vert Z_{kj} (A) \vert^{2} \right]= F_{kk}(A) = \sqrt{\pi}  c_{k} \sum_{\varepsilon = \pm 1} \delta_{\varepsilon \frac{\alpha}{2k+1}} (A) =0,
\end{align*}
and hence $\vert Z_{kj} (A) \vert=0$ a.s. so that the random measure $Z_{jk}$ is supported on the set $\left\{\frac{\alpha}{2k+1} , \, - \frac{\alpha}{2k+1} \right\}$, and the field can then be written as 
\begin{align*}
   & \xi_{\alpha} (g) =\sqrt{c_{-1}} \,b^{\C}_{\alpha} (z)+ \int_{\R} \sum_{k,j=0}^{\infty}\left\langle \pi_{\lambda} (g) \Phi_{k}, \Phi_{j} \right\rangle dZ_{jk} (\lambda)
   \\
   & = \sqrt{c_{-1}} \, b^{\C}_{\alpha} (z)+ \sum_{k,j=0}^{\infty} \sum_{\varepsilon = \pm 1} \left\langle \pi_{\varepsilon \frac{\alpha}{2k+1}} (g) \Phi_{k}, \Phi_{j} \right\rangle Z_{jk} \left( \left\{ \varepsilon \frac{\alpha}{2k+1} \right\} \right) 
   \\
   &= \sqrt{c_{-1}} \, b^{\C}_{\alpha} (z)+  \sum_{k,j=0}^{\infty}  \sum_{\varepsilon = \pm 1} W^{\alpha, \varepsilon}_{kj} e^{\alpha , \varepsilon}_{kj} (g),
\end{align*}
where 
\begin{align}\label{eq:W}
     & W^{\alpha, \varepsilon}_{kj}:= Z_{kj} \left( \left\{\varepsilon \frac{ \alpha}{2k+1} \right\} \right).
\end{align}
Note that $\{ W^{\alpha, \varepsilon}_{kj} , \; k,j \in \mathbb{N}_{0}, \, \varepsilon = \pm 1 \}$ are Gaussian random variables since $\xi_{\alpha}$ is a Gaussian field, and 
\begin{align*}
    \E \left[  W^{\alpha, \varepsilon}_{kj} \overline{W^{\alpha, \varepsilon'}_{k'j'} }   \right] = \sqrt{\pi}  c_{k} \delta_{k,k'}\delta_{j,j'}\delta_{\varepsilon,\varepsilon'}.
\end{align*}
Let us write 
\begin{align*}
    & W^{\alpha, \varepsilon}_{kj} = \sqrt{\sqrt{\pi}\,c_{k} }\, a^{\alpha, \varepsilon}_{kj},
\end{align*}
so that the random variables $\{a^{\alpha, \varepsilon}_{kj}, \, \varepsilon=\pm1,\,  k,j=0,1,\ldots  \}$ satisfy
\begin{align*}
    \E \left[ a^{\alpha, \varepsilon}_{kj} \overline{a^{\alpha, \varepsilon^{\prime}}_{k^{\prime}j^{\prime}}}\right]= \delta_{kk^{\prime}}\delta_{jj^{\prime}}\delta_{\varepsilon \varepsilon ^{\prime}},
\end{align*}
and hence 
\begin{align*}
    \xi_{\alpha} (g) = \sqrt{c_{-1}}\, b^{\C}_{\alpha} (z)+ \sum_{k,j=0}^{\infty} \sum_{\varepsilon = \pm 1}  \sqrt{\sqrt{\pi}\,c_{k} } a^{\alpha ,\varepsilon}_{kj} e^{\alpha, \varepsilon}_{kj} (g)=  \sqrt{c_{-1}}\, b^{\C}_{\alpha} (z)+\sum_{k=0}^{\infty} \sqrt{c_{k}} \, \xi^{k}_{\alpha} (g) ,
\end{align*}
where $\xi^{k}_{\alpha} (g)$ is given by \eqref{e.field.component}.

     $(3) \implies (1)$ Let $\xi_{\alpha}$ be the field defined by \eqref{representation.r.f.H}. Then, for $g=(z,t), \, h=(w,s) \in \Hei$
     \begin{align*}
    &  \E \left[ \xi_{\alpha} (g) \overline{\xi_{\alpha}(h)} \right] = c_{-1}\E \left[ b^{\C}_{\alpha} (z) \overline{b^{\C}_{\alpha}(w)} \right]+   \sum_{k,j=0}^{\infty} \sum_{\varepsilon= \pm 1}  \sqrt{\pi} c_{k}  e^{\alpha, \varepsilon}_{kj}  (g)\overline{e^{\alpha, \varepsilon}_{kj}  (h)} 
    \\
    &=2 c_{-1} J_{0} \left( \sqrt{\alpha} \vert z-w\vert \right) +\sum_{k=0}^{\infty} \sum_{\varepsilon= \pm 1}  \sqrt{\pi} c_{k} e^{\alpha, \varepsilon}_{kk}  (h^{-1}g) 
    \\
    &= 2 c_{-1}J_{0} \left( \sqrt{\alpha} \vert z-w\vert \right) + 2\sum_{k=0}^{\infty} c_{k} \varphi_{\alpha}^{k} (h^{-1} g) =: 2F^{(c)}_{\alpha}(h^{-1} g).
\end{align*}
         Evaluating the covariance at the identity $g=h=e$, and recalling that $J_0(0)=1$ and $\varphi_\alpha^k(e)=1$, we obtain $\E[|\xi_\alpha(e)|^2] = 2\sum_{k=-1}^\infty c_k = 2$ since $c \in \mathcal{O}^1$. The field is Gaussian since it is given as a series of i.i.d Gaussian random variables,
         this proves that $\xi_\alpha(g) \sim N_{\mathbb{C}}(0,2)$ for all $g \in \mathbb{H}$, fulfilling condition (i).
     Concerning (ii), the trajectories of $\xi_{\alpha}$ are smooth  a.s. by Theorem \ref{thm.smooth.field}. Let us now prove that they are generalized eigenfunctions. For a fixed $n\in \mathbb{N}$ we consider the approximated field 
     \begin{align*}
         \xi_{\alpha,n} (g) :=  b^{\C}_{\alpha} (z)+ \sum_{k,j=0}^{n} \sum_{\varepsilon= \pm 1} e^{\alpha, \varepsilon}_{kj}  (g)W^{\alpha, \varepsilon}_{kj},
     \end{align*}
     with $W^{\alpha, \varepsilon}_{kj}$ defined as in \eqref{eq:W},
     which satisfies 
     \begin{equation}\label{eqn.e.fun.approx}
         -\mathcal{L}   \xi_{\alpha,n} (g) = \alpha   \xi_{\alpha,n} (g),
     \end{equation}
     for all $n\in \mathbb{N}$ and $g\in \Hei$. Note that, for $n,m\in \mathbb{N}$
     \begin{align*}
         &\sup_{g\in\Hei} \E\left[ \left|  \xi_{\alpha,n+m} (g)-  \xi_{\alpha,m,} (g) \right|^{2} \right] \leq \sup_{g\in\Hei}\sum_{k,j=m}^{n+m} c_{k}\sum_{\varepsilon=\pm 1} \left|e^{\alpha, \varepsilon}_{kj}  (g)\right|^{2} \leq  \sup_{g\in\Hei} \sum_{k=m}^{n+m} c_{k}\sum_{\varepsilon=\pm 1}  \sum_{j=0}^{\infty}  \left|e^{\alpha, \varepsilon}_{kj}  (g)\right|^{2}
         \\
         & = \sup_{g\in\Hei} \sum_{k=m}^{n+m} c_{k}\sum_{\varepsilon=\pm 1}  \sum_{j=0}^{\infty}  \left| \left\langle \pi_{ \varepsilon\frac{\alpha}{2k+1}} (g) \Phi_{k} , \Phi_{j} \right\rangle \right|^{2} = 2  \sum_{k=m}^{n+m} c_{k} \rightarrow 0 \text{ as } m\rightarrow \infty,
     \end{align*}
     since $c\in \mathcal{O}^{1}$, proving that the series  $\xi_{\alpha ,n}$ converges uniformly to $\xi_{\alpha}$, where in the last equality we used Proposition \ref{prop.properties.e.functions}.
This proves that $\xi_{\alpha,n}$ converges to $\xi_\alpha$ uniformly in $L^2(\Omega)$. Since Theorem \ref{thm.smooth.field} guarantees that the sample paths of $\xi_\alpha$ are smooth a.s., the convergence holds also in the topology of $C^\infty(\mathbb{H})$ on compact sets almost surely. Therefore, one can take  the limit in  \eqref{eqn.e.fun.approx}  as $n \to \infty$, proving that $-\mathcal{L}\xi_\alpha = \alpha \xi_\alpha$ a.s. and hence (ii) follows. Lastly, to prove (iii) note that the covariance function $C(g,h)=2F^{(c)}_{\alpha}(h^{-1} g)$ only depends on  $h^{-1}g$, proving that the field is left-invariant. Moreover, $F^{(c)}_{\alpha} (\cdot)$ is a radial function  and hence the field is invariant under the action of $U(1)$.  

\end{proof}

\appendix 

\section{}\label{appendix.A}
The goal of this section is to prove Theorem \ref{thm.Bochner} and characterize positive-definite functions on locally compact groups of type I. Our proof is based on direct integrals of unitary irreducible representations and on a characterization first proven by Gelfand, Raikov, and Naimark \cite{Naimark59, GelfandRaikov43}.

\subsection{Direct integrals of representations}

 Let us start by recalling the direct integral decomposition of Hilbert spaces \cite[Ch. 7 Sec. 4]{Folland16}. Let $(B, \mathcal{B}, \nu)$ be a measure space and $\{H_{\beta} \}_{\beta \in B}$ be a collection of Hilbert spaces. A field on $B$ is a map $f : B \rightarrow \Pi_{\beta\in B} H_{\beta}$ such that $f(\beta) \in H_{\beta}$ for each $\beta$. A measurable field of Hilbert spaces over $B$ is a collection of Hilbert spaces $\{H_{\beta} \}_{\beta \in B}$ and fields $\{e_{j} \}_{j=1}^{\infty}$ such that 
\begin{align*}
    & (i)\; \; \beta \rightarrow \langle e_{j} (\beta), e_{k} (\beta) \rangle_{H_{\beta}} \text{ is measurable for all } j,k ,
    \\
    & (ii)\; \text{the linear span of } \{e_{j} (\beta) \}_{j=1}^{\infty} \text{ is dense in }H_{\beta} \text{ for each } \beta.
\end{align*}
The direct integral of the spaces $H_{\beta}$ with respect to $\nu$, denoted by 
\begin{align*}
    &\int_{B}^{\oplus} H_{\beta} d\nu (\beta),
\end{align*}
is the space of fields $f : B \rightarrow \Pi_{\beta\in B} H_{\beta}$ such that 
\begin{align*}
    & \Vert f \Vert^{2} := \int_{B} \Vert f(\beta) \Vert^{2}_{\beta} d\nu (\beta)< \infty,
\end{align*}
where $\Vert \cdot \Vert_{\beta}$ is the norm induced by the inner product on $H_{\beta}$. It can be proven that the direct integral is a Hilbert space with respect to the inner product 
\begin{align*}
    & \langle f , g\rangle:= \int_{B} \langle f(\beta) , g(\beta) \rangle_{H_{\beta}} d\nu (\beta).
\end{align*}
A measurable field of operators is a map 
\begin{align*}
    T : B \rightarrow  \Pi_{\beta\in B}\mathcal{L} (H_{\beta}),
\end{align*}
where $\mathcal{L} (H_{\beta})$ denotes the space of linear operators on $H_{\beta}$, such that 
\begin{align*}
    \beta \rightarrow T(\beta) f(\beta) \in H_{\beta}
\end{align*}
is measurable for each $\beta$. A measurable field of operators defines an operator $\int_{B}^{\oplus} T(\beta) d\nu (\beta)$ acting on $\int_{B}^{\oplus} H_{\beta} d\nu (\beta)$. Let $B_{1} \in \mathcal{B}$ be a measurable set. Then the operator 
\begin{align*}
    & \int_{B_{1}}^{\oplus} T(\beta) d\nu (\beta) \text{ acts on }  \int_{B_{1}}^{\oplus} H_{\beta} d\nu (\beta),
\end{align*}
which is a Hilbert subspace of $\int_{B}^{\oplus} H_{\beta} d\nu (\beta)$. We can then extend $\int_{B_{1}}^{\oplus} T(\beta) d\nu (\beta)$ by letting 
\begin{align*}
    & \int_{B_{1}}^{\oplus} T(\beta) d\nu (\beta)v := 0 , & v \in \left(  \int_{B_{1}}^{\oplus} T(\beta) d\nu (\beta) \right)^{\perp},
\end{align*}
where the latter denotes the orthogonal complement of $\int_{B_{1}}^{\oplus} T(\beta) d\nu (\beta)$ in $\int_{B}^{\oplus} T(\beta) d\nu (\beta)$. Note that for any two measurable sets $B_{1}, B_{2}$ one has that 
\begin{align}\label{eq.intersection}
    & \int_{B_{1}}^{\oplus} T(\beta) d\nu (\beta) \int_{B_{2}}^{\oplus} T(\beta) d\nu (\beta) = \int_{B_{1}\cap B_{2}}^{\oplus} T(\beta) d\nu (\beta),
\end{align}
and 
\begin{align*}
    & \int_{B_{1}}^{\oplus} T(\beta) d\nu (\beta) \int_{B_{2}}^{\oplus} T(\beta) d\nu (\beta) =0,
\end{align*}
whenever $B_{1} \cap B_{2} = \emptyset$.

Now, let $\pi$ be a unitary representation of a  separable, locally compact group of type I $G$
 on a Hilbert space $H_{\pi}$. By \cite[p. 159] {Mackey57} and \cite{Yaglom60, Guichardet60}, it follows that  $\pi$ can be decomposed as 
\begin{align*}
    & \pi = \bigoplus_{\ell}\pi^{\ell} ,  & H_{\pi} =  \bigoplus_{\ell}H^{\ell},
\end{align*}
where each $\pi^{\ell}$ is a multiplicity free representation acting on the Hilbert space $H^{\ell}$. Moreover, each $H^{\ell}$ can be represented as a direct integral of Hilbert spaces $H_{\lambda}$ corresponding to non-equivalent, unitary, irreducible representations $\pi_{\lambda}$'s, and $\pi^{\ell}$ as a direct integrals of  $\pi_{\lambda}$'s, that is, 
\begin{align}\label{eqn.decomposition}
   & H^{\ell} = \int_{\widehat{G}}^{\oplus} H_{\lambda} d\sigma^{\ell} (\lambda) & \pi^{\ell} = \int_{\widehat{G}}^{\oplus} \pi_{\lambda} d\sigma^{\ell} (\lambda),
\end{align}
where $\sigma^{\ell}$ is a measure on $\widehat{G}$, see \cite[Section 10] {Mackey57}.  For a fixed orthonormal basis $\{e_{i} (\lambda) \}_{i=1}^{\infty}$  for  $H_{\lambda}$, set
\begin{align*}
    & \tilde{e}^{\ell}_{i} : = \int_{\widehat{G}}^{\oplus} e_{i} (\lambda) d\sigma^{\ell} (\lambda),
\end{align*}
so that $\{\tilde{e}^{\ell}_{i}  \}_{i=1}^{\infty}$ is an orthonormal basis for  $H^{\ell}$. Lastly, let us recall the notion of positive operator-valued measures (POVMs). A POVM  $F$ on a measurable space $(B, \mathcal{B})$ is a function 
\begin{align*}
    F: \mathcal{B} \rightarrow \mathcal{L} (H),
\end{align*}
where $\mathcal{L} (H)$ denotes the space of linear operators on some Hilbert space $H$, such that $F(A)$ is a positive operator for all $\in \mathcal{B}$ and the map 
\begin{align*}
    &\mathcal{B}  \longrightarrow \R,
    \\
    & A\rightarrow \langle F(A) v,w\rangle,
\end{align*}
is a non-negative countably additive measure for all $v,w\in H$.

\subsection{Positive-definite functions on locally compact groups of type I}

A first characterization of positive-definite functions on locally compact groups of type I was proven in \cite{Naimark59, GelfandRaikov43}, in the following form. 

\begin{thm}\label{thm.Bochner1}
    Let $G$ be a separable, locally compact group of type I. A function $B: G \longrightarrow \mathbb{C}$ is positive-definite if and only if it can be represented as 
    \begin{equation}\label{eqn.pos.def.fun1}
        B(g)=\langle \pi (g) \xi^{0}, \xi^{0} \rangle_{H},
    \end{equation}
    where $\pi$ is a unitary representation of $G$ acting on the Hilbert space $H$ and $\xi^{0} \in H$ is a fixed vector.
\end{thm}

By means of direct integral of unitary irreducible representations, and POVMs, we have the following version of Theorem \ref{thm.Bochner1}.

\begin{thm}\label{thm.Bochner.appendix}
    Let $G$ be a separable, locally compact group of type I. A function $B: G \longrightarrow \mathbb{C}$ is positive definite if and only if it can be represented as 
    \begin{equation}\label{eqn.pos.def.fun.appendix}
        B(g) = \operatorname{Tr} \left( \pi (g)  F(\widehat{G} ) \right),
    \end{equation}
    where $\widehat{G}$ denotes the unitary dual of $G$,  $\pi$ is a unitary representation of $G$ acting on the Hilbert space $H_{\pi}$, and $F$ is a positive operator-valued measure over $\widehat{G}$, whose values are  Hermitian non-negative definite operators on the representation space $H_{\pi}$, such that 
    \begin{align}\label{e.condition.pos.def.appendix}
       \operatorname{Tr} \left( F(\widehat{G}) \right) = \Vert \xi_0 \Vert_{H_{\pi}}^{2} < \infty,
    \end{align}
    where $\xi_{0} \in H_{\pi}$ is given by \eqref{eqn.pos.def.fun1}.
\end{thm}

The positive operator-valued measure $F$ in Theorem \ref{thm.Bochner.appendix} is explicitly given in terms of the  vector $\xi_{0} \in H_{\pi}$ and the multiplicity free representations in the decomposition \eqref{eqn.decomposition} of $\pi$, as shown in the following proof.

\begin{proof}
    By Theorem \ref{thm.Bochner1} we know that $B: G \longrightarrow \mathbb{C}$ is positive definite if and only if 
    \begin{align*}
        &  B(g) = \langle \pi (g) \xi^{0}, \xi^{0} \rangle_{H_{\pi}},
    \end{align*}
    for some unitary representation $\pi$ acting on $H_{\pi}$ and some vector $\xi^{0} \in H_{\pi}$. It is then enough to prove that 
    \begin{align*}
        & \langle \pi (g) \xi^{0}, \xi^{0} \rangle_{H_{\pi}} = \operatorname{Tr} \left( \pi (g)  F(\widehat{G} ) \right),
    \end{align*}
    for some  POVM $F$ over $\widehat{G}$. By \eqref{eqn.decomposition}, let us decompose $\xi^{0}\in H_{\pi}$ as 
\begin{align*}
    & \xi^{0} =  \sum_{\ell} \xi_{\ell}^{0} , & \xi^{0}_{\ell} \in H^{\ell},
    \\
    & \xi^{0}_{\ell}=  \int_{\widehat{G}}^{\oplus} \xi^{0} (\lambda) d\sigma^{\ell} (\lambda), &  \xi^{0} (\lambda) \in H_{\lambda},
    \\
    & \xi^{0} (\lambda) = \sum_{i} \xi^{0}_{i} (\lambda) e_{i} (\lambda).
\end{align*}
For a fixed $\ell$, that is, for a fixed multiplicity free representation $\pi^{\ell}$, let us introduce a POVM $F^{\ell}$ on $\widehat{G}$ in the following way. For each $A\in \mathcal{B}(\widehat{G})$, the $\sigma$-algebra generated by Borel subsets of $\widehat{G}$, we let $F^{\ell} (A)$ be the operator on the Hilbert space 
\begin{align*}
    &  \int_{A}^{\oplus} H_{\lambda}(g) d\sigma^{\ell} (\lambda),
\end{align*}
defined by
\begin{align}\label{eqn.povm.special}
    & F^{\ell} (A) \tilde{e}^{\ell,A}_{i} :=  \int_{A}^{\oplus} \overline{\xi^{0}_{i} (\lambda)} \xi^{0} (\lambda)  d\sigma^{\ell} (\lambda),
\end{align}
where $\xi^{0}$ is as above, and
\begin{align*}
    & \tilde{e}^{\ell,A}_{i} : = \int_{A}^{\oplus} e_{i} (\lambda) d\sigma^{\ell} (\lambda),
\end{align*}
so that $\{\tilde{e}^{\ell,A}_{i}  \}_{i=1}^{\infty}$ is an orthonormal basis for  $\int_{A}^{\oplus} H_{\lambda}(g) d\sigma^{\ell} (\lambda)$. Let us denote by $F^{\ell}_{ij}$ the measure given by 
\begin{align*}
    & A\rightarrow F^{\ell}_{ij} (A):= \langle F^{\ell}(A)  \tilde{e}^{\ell,A}_{i} ,  \tilde{e}^{\ell,A}_{j}\rangle = \int_{A} \overline{\xi^{0}_{i} (\lambda)} \langle  \xi^{0} (\lambda) , e_{j} (\lambda) \rangle  d\sigma^{\ell} (\lambda) = \int_{A} \overline{\xi^{0}_{i} (\lambda)}\xi^{0}_{j} (\lambda)  d\sigma^{\ell} (\lambda) ,
\end{align*}
and hence define $F^{\ell}_{ij} (d\lambda) :=\overline{\xi^{0}_{i} (\lambda)}\xi^{0}_{j} (\lambda)  d\sigma^{\ell} (\lambda)$. Note that the family of measures $F^{\ell}_{ij}$ satisfies 
\begin{align}\label{eqn.hermitian.ish}
    F^{\ell}_{ij} = \overline{F^{\ell}_{ji}}.
\end{align}
We can then define an operator-valued measure over $\widehat{G}$ acting on $H_{\pi}$ by 
\begin{align*}
    & F(A) v:= \sum_{\ell} F^{\ell} (A) v_{\ell} , & \text{ for any } v= \sum_{\ell} v_{\ell} \in H_{\pi}.
\end{align*}
Thus,
\begin{align*}
    & \operatorname{Tr} \left( \pi (g)  F(\widehat{G} ) \right) = \sum_{\ell}  \operatorname{Tr} \left( \pi^{\ell} (g) F^{\ell} (\widehat{G}) \right)=  \sum_{\ell} \operatorname{Tr} \left( \int_{\widehat{G}}^{\oplus} \pi_{\lambda} (g)d\sigma^{\ell} (\lambda) \cdot F^{\ell}(\widehat{G} ) \right)
    \\
    & = \sum_{\ell, j} \langle  \int_{\widehat{G}}^{\oplus} \pi_{\lambda} (g)d\sigma^{\ell} (\lambda) \cdot F^{\ell}(\widehat{G} ) \tilde{e}^{\ell}_{j} , \tilde{e}^{\ell}_{j} \rangle_{H^{\ell}} =   \sum_{\ell, j} \langle  F^{\ell}(\widehat{G} ) \tilde{e}^{\ell}_{j} ,   \int_{\widehat{G}}^{\oplus} \pi_{\lambda} (g^{-1})d\sigma^{\ell} (\lambda) \tilde{e}^{\ell}_{j} \rangle_{H^{\ell}}
    \\
    & =    \sum_{\ell, j} \langle \int_{\widehat{G}}^{\oplus} \overline{\xi^{0}_{j} (\lambda)} \xi^{0} (\lambda)  d\sigma^{\ell} (\lambda)  ,   \int_{\widehat{G}}^{\oplus} \pi_{\lambda} (g^{-1}) e_{j}(\lambda)d\sigma^{\ell} (\lambda)  \rangle_{H^{\ell}}
    \\
    &  =    \sum_{\ell, j} \int_{\widehat{G}}  \overline{\xi^{0}_{j} (\lambda)} \langle \xi^{0} (\lambda) , \pi_{\lambda} (g^{-1}) e_{j}(\lambda) \rangle_{H_{\lambda}}d\sigma^{\ell} (\lambda)
    \\
    & =      \sum_{\ell, j,i} \int_{\widehat{G}}  \xi^{0}_{i} (\lambda) \overline{\xi^{0}_{j} (\lambda)} \langle   \pi_{\lambda} (g)  e_{i}(\lambda)  ,e_{j}(\lambda) \rangle_{H_{\lambda}} d\sigma^{\ell} (\lambda)
    \\
    &= \sum_{\ell,i,j}  \int_{\widehat{G}}\pi^{ij}_{\lambda}(g) \xi_{i}^{0} (\lambda) \overline{\xi_{j}^{0} (\lambda)}    d\sigma^{\ell} (\lambda) = \sum_{\ell,i,j}  \int_{\widehat{G}}\pi^{ij}_{\lambda}(g) F^{\ell}_{ji} (d\lambda),
\end{align*}
where $\pi^{ij}_{\lambda}(g) :=  \langle  \pi_{\lambda}(g)e_{i} (\lambda) , e_{j} (\lambda)  \rangle_{H_{\lambda}}$ denotes the matrix element. On the other hand, 
\begin{align*}
    & B(g)= \langle \pi (g) \xi^{0} , \xi^{0} \rangle_{H_{\pi}}  =\sum_{\ell} \langle \pi^{\ell} (g) \xi^{0}_{\ell} , \xi^{0}_{\ell} \rangle_{H^{\ell}} 
    \\
    & =  \sum_{\ell} \langle  \int_{\widehat{G}}^{\oplus} \pi_{\lambda}(g)\xi^{0} (\lambda) d\sigma^{\ell} (\lambda), \int_{\widehat{G}}^{\oplus} \xi^{0} (\lambda) d\sigma^{\ell} (\lambda) \rangle_{H^{\ell}}
    \\
    & =\sum_{\ell} \int_{\widehat{G}} \langle  \pi_{\lambda}(g)\xi^{0} (\lambda) , \xi^{0} (\lambda)  \rangle_{H_{\lambda}}d\sigma^{\ell} (\lambda)
    \\
    &= \sum_{\ell,i,j}  \int_{\widehat{G}}\pi^{ij}_{\lambda}(g) \xi_{i}^{0} (\lambda) \overline{\xi_{j}^{0} (\lambda)}    d\sigma^{\ell} (\lambda) = \sum_{\ell,i,j}  \int_{\widehat{G}}\pi^{ij}_{\lambda}(g) F^{\ell}_{ji} (d\lambda),
\end{align*}
 In particular, we have that 
\begin{align*}
    B(g) = \sum_{\ell,i,j}  \int_{\widehat{G}}\pi^{ij}_{\lambda}(g) \xi_{i}^{0} (\lambda) \overline{\xi_{j}^{0} (\lambda)}    d\sigma^{\ell} (\lambda),
\end{align*}
as in  \cite[p. 599]{Yaglom60}. Moreover, note that the operator $F(A)$ is Hermitian in light of \eqref{eqn.hermitian.ish}, and that
\begin{align*}
    &  \operatorname{Tr} \left( F(\widehat{G}) \right)  = \sum_{\ell} \operatorname{Tr} \left( F^{\ell}(\widehat{G}) \right) = \sum_{\ell, i} \langle F^{\ell} (\widehat{G}) \tilde{e}^{\ell}_{i},\tilde{e}^{\ell}_{i} \rangle_{H^{\ell}} 
    \\
    & = \sum_{\ell, i} \langle  \int_{\widehat{G}}^{\oplus} \overline{\xi^{0}_{i}(\lambda)} \xi^{0} (\lambda) d\sigma^{\ell} (\lambda)   , \int_{\widehat{G}}^{\oplus} e_{i}(\lambda) d\sigma^{\ell} (\lambda) \rangle_{H^{\ell}} 
    \\
    &= \sum_{\ell , i} \int_{\widehat{G}}  \langle  \overline{\xi^{0}_{i}(\lambda)} \xi^{0} (\lambda), e_{i}(\lambda) \rangle_{H_{\lambda}} d\sigma^{\ell} (\lambda)
    \\
    & = \sum_{\ell,i} \int_{\widehat{G}} \vert \xi_{i}^{0} (\lambda)\vert^{2} d\sigma^{\ell} (\lambda)= \sum_{\ell} \int_{\widehat{G}} \Vert \xi^{0} (\lambda)\Vert_{H_{\lambda}}^{2} d\sigma^{\ell} (\lambda)
    \\
    & = \sum_{\ell} \Vert \xi^{0}_{\ell} \Vert^{2}_{H^{\ell}} = \Vert \xi^{0} \Vert_{H_{\pi}}^{2} <\infty,
\end{align*}
which completes the proof.
\end{proof}

Finally, we can write  \eqref{eqn.pos.def.fun}, \eqref{e.condition.pos.def}, \eqref{e.stationary.r.f.group}, and \eqref{eqn.cov.random.op.measure} explicitly in terms of a basis for $H_{\pi}$ in the following way. If $\{e_{i} (\lambda) \}_{i=1}^{\infty}$ is an orthonormal basis for  $H_{\lambda}$, then from the proof of Theorem \ref{thm.Bochner.appendix} we have that 
\begin{align}\label{eqn.cov.components}
    &B(g)= \sum_{\ell, i,j}  \int_{\widehat{G}} \pi^{ij}_{\lambda} (g) F^{\ell}_{ji} (d\lambda) =  \sum_{i,j}  \int_{\widehat{G}} \pi^{ij}_{\lambda} (g) F_{ji} (d\lambda),
\end{align}
where $\pi^{ij}_{\lambda}(g) :=  \langle  \pi_{\lambda}(g)e_{i} (\lambda) , e_{j} (\lambda)  \rangle_{H_{\lambda}}$ denotes the matrix element and 
\begin{align*}
    F_{ij} (A):= \sum_{\ell} F^{\ell}_{ij} (A),
\end{align*}
where $\{ F_{ij}\}_{i,j}$ is a sequence of measures such that $F_{ij}= \overline{F}_{ji}$ and 
\begin{align*}
    \sum_{i} F_{ii} (\overline{ G})= \sum_{\ell,i} F^{\ell}_{ii} (\overline{ G})<\infty.
\end{align*}
Lastly, the corresponding field  \eqref{e.stationary.r.f.group} can be written as 
\begin{align*}
    & \xi(g) = \int_{\widehat{G}} \sum_{i,j=0}^{\infty} \pi_{\lambda}^{ij}(g) dZ_{ji} (\lambda),
\end{align*}
where $\{Z_{ij}\}_{i,j}$ is a family of random measures on $\widehat{G}$ such that 
\begin{align*}
    & \E \left[  Z_{ij}(A) \overline{Z_{kl}(B)} \right]= \delta_{jl} F_{ik}(A \cap B) ,
\end{align*}
see \cite[eqs. (2.26'), (2.33'), and (2.35')]{Yaglom60}.

\begin{acknowledgement}
The author A.P.T. is a member of INdAM-GNAMPA and acknowledges the hospitality of the University of Melbourne, where part of the work was accomplished during her visit.
\end{acknowledgement}

\bibliography{references}
\bibliographystyle{amsplain}

\end{document}